\documentclass[11pt]{article}
\usepackage{graphicx,epsfig}
\usepackage{fancyhdr,fancybox,float}
\usepackage[title]{appendix}
\usepackage{indentfirst}
\usepackage{verbatim}
\usepackage[sort&compress, numbers]{natbib}
\usepackage{geometry}
\usepackage{extarrows,chemarrow,xypic}
\usepackage{color}
\usepackage{caption}
\usepackage{bbding}
\usepackage{microtype}
\usepackage{amssymb}
\usepackage{makecell}
\usepackage{booktabs}
\usepackage{arydshln}
\usepackage{multirow}
\usepackage{array}
\usepackage{mathtools}
\DisableLigatures[f]{encoding = *, family = *}
\usepackage{times}

\usepackage{amsthm}
\usepackage{mathrsfs}
\usepackage{bm}
\usepackage{bbm}
\usepackage{hyperref}

\usepackage{subfigure}
\usepackage[T1]{fontenc}
\usepackage{float}

\newcommand{\paperfont}{\fontsize{10pt}{1.1\baselineskip}\selectfont}
\definecolor{blue}{RGB}{0,50,200}
\definecolor{red}{RGB}{170,33,44}

\begin{document}

\theoremstyle{definition}
\makeatletter
\thm@headfont{\bf}
\makeatother
\newtheorem{definition}{Definition}
\newtheorem{example}{Example}
\newtheorem{theorem}{Theorem}
\newtheorem{lemma}{Lemma}
\newtheorem{corollary}{Corollary}
\newtheorem{remark}{Remark}
\newtheorem{proposition}{Proposition}

\lhead{}
\rhead{}
\lfoot{}
\rfoot{}

\renewcommand{\refname}{References}
\renewcommand{\figurename}{Figure}
\renewcommand{\tablename}{Table}
\renewcommand{\proofname}{Proof}
\def\code#1{\texttt{#1}}

\newcommand{\diag}{\mathrm{diag}}
\newcommand{\tr}{\mathrm{tr}}
\newcommand{\re}{\mathrm{Re}}
\newcommand{\one}{\mathbbm{1}}
\newcommand{\Pnum}{\mathbb{P}}
\newcommand{\Enum}{\mathbb{E}}
\newcommand{\Rnum}{\mathbb{R}}
\newcommand{\Nnum}{\mathbb{N}}
\newcommand{\Znum}{\mathbb{Z}}
\newcommand{\dnum}{\mathrm{d}}
\newcommand{\hyper}{{}_2F_1}
\newcommand{\confl}{{}_1F_1}

\title{\textbf{Long-time behavior of reducible stochastic linear reaction networks: a spectral and structural classification}}
\author{Bingjie Wu$^{1}$, \quad Hao Kang$^{1}$, \quad Chen Jia$^{2,3,*}$\\
\footnotesize $^1$ Center for Applied Mathematics and KL-AAGDM, Tianjin University, Tianjin 300072, China. \\
\footnotesize $^2$ School of Mathematics, Harbin Institute of Technology, Harbin 150001, China. \\
\footnotesize $^3$ Beijing Research Institute, Harbin Institute of Technology, Beijing 100144, China. \\
\footnotesize $^*$ Correspondence: chenjia@hit.edu.cn}
\date{}
\maketitle
\thispagestyle{empty}

\paperfont
\captionsetup{font=footnotesize}

{\abstract Exponential ergodicity of stochastic reaction networks has attracted considerable attention in recent years [SIAM J. Appl. Dyn. Syst. 24, 1668-1710 (2025)]. Here, we provide a structural classification of the long-time behavior of reducible stochastic linear reaction networks under the $L^1$-Wasserstein and total variation distances. The classification is determined by the maximal eigenvalue $\lambda_{\max}$ of the first-order influx matrix $A$, the position of the zero-order influx vector $b$ relative to the left nullspace of $A$, and the conservation-law structure of the network restricted to the persistent species. We first prove that every stochastic linear reaction network is non-explosive. In the stable regime $\lambda_{\max}<0$, the process converges exponentially fast to a unique stationary distribution. In the critical regime $\lambda_{\max}=0$, when the zero eigenvalue of $A$ is semisimple and $b$ is orthogonal to the left nullspace of $A$, exponential convergence occurs if and only if a regularity condition on the restricted network is satisfied. If these two spectral conditions hold but the regularity condition fails, convergence occurs only in total variation and is non-exponential. In the divergent regimes, no closed irreducible positive recurrent class can contain an interior state, and any stationary distribution, if it exists, must be supported on the boundary. Moreover, by employing a coupling method, we obtain the optimal convergence rate in the exponentially convergent cases.}

\section{Introduction}\label{sec:introduction}
Stochastic chemical reaction networks provide a standard mesoscopic framework for describing biochemical and intracellular processes, in which molecular copy numbers evolve according to a continuous-time Markov chain on a high-dimensional nonnegative integer lattice \cite{anderson2015stochastic}. They have been widely used in chemistry, molecular biology, systems biology, synthetic biology, infectious disease dynamics, and ecology, particularly in settings where intrinsic stochastic fluctuations and low-copy-number effects cannot be neglected. When molecule numbers become sufficiently large, stochastic fluctuations can often be neglected according to the law of large numbers, and the evolution of reaction networks can be approximated by deterministic reaction-rate equations \cite{kampen1961power, kurtz1978strong}. The systematic study of the deterministic and stochastic dynamics of reaction networks, together with their relationships to network topology, constitutes the foundation of modern reaction network theory \cite{horn1972general, feinberg1972complex, feinberg1987chemical, craciun2013persistence}.

Linear reaction networks form an important subclass in which every reaction is either zeroth-order or involves a single reactant molecule \cite{xu2026first}. This subclass includes monomolecular reaction networks \cite{gans1960open, darvey1966stochastic} and many important models of stochastic gene expression \cite{paulsson2005models} as special cases. Owing to their affine propensity functions, a number of exact analytical results have been obtained for stochastic linear reaction networks. Under mild conditions, it is well known that a closed monomolecular reaction network, consisting only of reactions of the form $S_i\rightarrow S_j$, admits a multinomial stationary distribution, whereas an open monomolecular reaction network, consisting of synthesis reactions $\varnothing\rightarrow S_i$, degradation reactions $S_i\rightarrow \varnothing$, and conversion reactions $S_i\rightarrow S_j$, admits a product-form Poisson stationary distribution \cite{kelly1979reversibility, gadgil2005stochastic, heuett2006grand, jahnke2007solving}. These results have been further extended to general stochastic reaction networks that are complex balanced. In fact, the stationary distribution of a complex-balanced stochastic reaction network is also given by a product-form Poisson-like distribution \cite{anderson2010product, cappelletti2016product, pascual2022algebraic}. When complex balance is not satisfied, exact stationary distributions have been analytically derived for hierarchic and more general first-order reaction networks \citep{reis2018general, li2021steady}. In the context of stochastic gene expression, the stationary distributions of mRNA and protein copy numbers have been exactly characterized in terms of hypergeometric functions \cite{peccoud1995markovian, shahrezaei2008analytical, bokes2012exact, zhou2012analytical, pendar2013exact, jia2023analytical}.

Although many analytical results have been obtained for stochastic linear reaction networks, they do not address the existence of a unique stationary distribution (positive recurrence) or whether convergence to the stationary distribution occurs at an exponential rate (exponential ergodicity). The Foster-Lyapunov theory provides a fundamental framework for establishing these properties for continuous-time Markov chains \cite{meyn1993stability, down1995exponential}. In the reaction-network setting, structural criteria for positive recurrence and exponential ergodicity have been recently developed \cite{anderson2018some, briat2020ergodicity, anderson2022mixing, anderson2025new}. In particular, for linear and asymptotically linear reaction networks, exponential ergodicity has been established under irreducibility assumptions \cite{gupta2014scalable, cappelletti2025stability}; exponential ergodicity has also been obtained for first-order reaction networks satisfying an endotactic structural condition \cite{xu2026exponential}. However, existing results on exponential ergodicity of linear reaction networks rely on additional irreducibility or structural assumptions, and a complete theoretical classification of the long-time behavior of reducible systems remains unavailable. Moreover, most existing studies on positive recurrence and exponential ergodicity of reaction networks are based on convergence in the total variation distance. Results on convergence in the Wasserstein distance and the optimal convergence rate remain largely unexplored. Compared with the total variation distance, the Wasserstein distance is more closely connected to optimal transport theory and is better suited for characterizing sharp exponential convergence rates \cite{chen2020limit, gaillard2026quantitative}.

In the present paper, we provide a structural classification of the long-time convergence behavior of reducible stochastic linear reaction networks under both the $L^1$-Wasserstein and total variation distances. Our classification is based on (i) the spectrum of the first-order influx matrix $A$, (ii) the position of the zero-order influx vector $b$ relative to the left nullspace of $A$, and (iii) whether every left null vector of $A$ corresponds to a conservation law when the reaction network is restricted to the set of persistent species. Based on these features, we identify the conditions characterizing three types of long-time behaviors: (i) exponential convergence to a stationary limiting distribution, (ii) non-exponential convergence to a stationary limiting distribution, and (iii) the absence of positive recurrent states in the interior of the state space. Moreover, using coupling arguments, we obtain the optimal convergence rate in the exponentially convergent cases. Finally, we present a variety of examples, including monomolecular reaction systems, gene expression systems, birth-death processes, and other representative models, to illustrate the main results.

The structure of the present paper is organized as follows. In Section~\ref{sec:prelim}, we recall the deterministic and stochastic dynamics of linear reaction networks and characterize the long-time behavior of the deterministic system. Section~\ref{sec:main} presents the main results of the paper. Specifically, we prove that stochastic linear reaction networks are non-explosive and provide a structural classification of the long-time convergence behavior of reducible stochastic linear reaction networks under the $L^1$-Wasserstein and total variation distances. The detailed proofs of these results are given in Sections~\ref{sec:stable_proof}-\ref{sec:critical_irregular_proof}. Finally, Section~\ref{sec:conclusion} summarizes the conclusions and discusses directions for future research.

\section{Model}\label{sec:prelim}
Throughout the paper, we denote by $\mathbb{N}=\{0,1,2,\cdots\}$ the set of non-negative integers, $\mathbb{R}_{\ge0}$ the set of non-negative real numbers, and $\mathbb{C}$ the set of complex numbers. The real part of a complex number $x\in\mathbb{C}$ is denoted by $\operatorname{Re}(x)$. For a real vector $v$, let $\|v\|_1=\sum_i|v_i|$ denote its standard $\ell^1$-norm. The notation $v>0$ ($v\geq0$) means that every component of $v$ is strictly positive (non-negative).

\subsection{Linear reaction networks}\label{subsec:first_order}
Consider a chemical reaction system consisting of $N$ chemical species, $S_1,\cdots,S_N$, and $M$ chemical reactions
\begin{equation*}
R_j:\;\mu_1^jS_1+\cdots+\mu_N^jS_N\rightarrow\nu_1^jS_1+\cdots+\nu_N^jS_N,\quad j = 1,\cdots,M,
\end{equation*}
where $\mu_i^j$ and $\nu_i^j$ denote the numbers of molecules of species $S_i$ consumed and produced, respectively, in a single occurrence of reaction $R_j$. For simplicity, we write $\mu_j=(\mu_1^j,\cdots,\mu_N^j)\in\Nnum^N$ and $\nu_j=(\nu_1^j,\cdots,\nu_N^j)\in\Nnum^N$, which are referred to as the \emph{source and product complexes}, respectively. Source and product complexes are collectively referred to as chemical complexes, or simply \emph{complexes}. The vector $w_j=\nu_j-\mu_j$ is called the \emph{reaction vector} of $R_j$, and the integer $\|\mu_j\|_1=\sum_{i=1}^N\mu_i^j$ is called the \emph{order} of $R_j$.

\begin{definition}
Let $\mathcal{S}=\{S_i\}_{i=1,\cdots,N}$ be the collection of all chemical species, let $\mathcal{C}=\{\mu_j,\nu_j\}_{j=1,\cdots,M}$ be the collection of all complexes, and let $\mathcal{R}=\{R_j\}_{j=1,\cdots,M}$ be the collection of all reactions. The triplet $(\mathcal{S},\mathcal{C},\mathcal{R})$ is referred to as a \emph{chemical reaction network}. A reaction network is called \emph{linear} or \emph{first-order} if it consists only of zero-order and first-order reactions, that is, if $\|\mu_j\|_1=0$ or $1$ for all $j=1,\cdots,M$.
\end{definition}

\begin{definition}\label{subspace}
The vector space
\begin{equation*}
\mathcal{P}=\textrm{span}(w_1,\cdots,w_M)
\end{equation*}
spanned by all reaction vectors is called the \emph{stoichiometric subspace} of the reaction network. For any $x\in\mathbb{R}_{\geq0}^N$, the set
\begin{equation*}
\mathcal{P}_x=(x+\mathcal{P})\cap\mathbb{R}_{\geq0}^N
\end{equation*}
is called the \emph{stoichiometric compatibility class} containing $x$.
\end{definition}

Concretely, a linear reaction network can be written in the following general form:
\begin{equation}\label{eq:first_order_system}
\begin{split}
R_{0j}:&\; \varnothing\xrightarrow{k_{0j}}\nu^{0j}_1S_1+\nu^{0j}_2S_2+\cdots+\nu^{0j}_N S_N, \quad j = 1,\cdots, r_0,\\
R_{ij}:&\; S_i\xrightarrow{k_{ij}}\nu^{ij}_1S_1+\nu^{ij}_2S_2+\cdots+\nu^{ij}_NS_N, \quad i = 1,\cdots,N,\;j = 1,\cdots, r_i,
\end{split}
\end{equation}
where $R_{0j}$, $j=1,\cdots,r_0$, are all zero-order reactions in the system, $R_{ij}$, $j=1,\cdots,r_i$, are all first-order reactions with reactant $S_i$, and $k_{0j}$ and $k_{ij}$ are the corresponding positive rate constants. Let
\begin{equation*}
\nu_{0j} = (\nu^{0j}_1,\cdots,\nu^{0j}_N) \in \mathbb{N}^N, \quad
\nu_{ij} = (\nu^{ij}_1,\cdots,\nu^{ij}_N) \in \mathbb{N}^N,
\end{equation*}
denote the product complexes of the zero-order and first-order reactions, respectively, and let $e_i=(0,\cdots,1,\cdots,0)$ denote the $i$th standard basis vector of $\mathbb{R}^N$. By Definition~\ref{subspace}, the stoichiometric subspace of the linear reaction network is given by
\begin{equation*}
\mathcal{P} = \operatorname{span}(\{\nu_{0j},\;j=1,\cdots,r_0\}\cup\{\nu_{ij}-e_i,\;i=1,\cdots,N,\;j=1,\cdots,r_i\}),
\end{equation*}
where $\nu_{0j}$ are the reaction vectors of the zero-order reactions and $\nu_{ij}-e_i$ are the reaction vectors of the first-order reactions.

Note that linear reaction systems include synthesis reactions $\varnothing \rightarrow S_i$, degradation reactions $S_i\rightarrow \varnothing$, conversion reactions $S_i\rightarrow S_j$, catalytic reactions $S_i\rightarrow S_i+S_j$, and splitting reactions $S_i\rightarrow S_j+S_k$ as special cases. Because of their generality and flexibility, they are widely used to model a broad range of stochastic phenomena in biology, physics, and chemistry.

\subsection{Stochastic dynamics of linear reaction networks}\label{subsec:cme_invariant}
We now focus on the stochastic dynamics of a linear reaction network under mass-action kinetics \cite{anderson2015stochastic}. The microstate of the stochastic system is represented by a vector $n=(n_1,\cdots,n_N)\in\mathbb{N}^N$, where $n_i$ denotes the molecule number of species $S_i$. The evolution of the stochastic system is governed by a continuous-time Markov chain (also called a Markov jump process) $(X_t)_{t\ge0}$ on the state space $E=\mathbb{N}^N$, where $X_t=(X_t^1,\cdots,X_t^N)$ and $X_t^i$ represents the molecule number of species $S_i$ at time $t$. The propensity functions of the zero-order and first-order reactions are given by
\begin{equation*}
\begin{gathered}
f_{R_{0j}}(n)=k_{0j},\quad j=1,\cdots,r_0,\\
f_{R_{ij}}(n)=k_{ij}n_i,\quad i=1,\cdots,N,\quad j=1,\cdots,r_i.
\end{gathered}
\end{equation*}
We emphasize that in a linear reaction network, multiple reactions may have the same reaction vector (for example, the zero-order reaction $\varnothing\rightarrow S_2$ and the first-order reaction $S_1\rightarrow S_1+S_2$). Hence the transition rates of the process $X_t$ are given by
\begin{equation*}
q_{n,n'}=\sum_{\substack{R\in\mathcal{R}:\;n+w_R=n'}}f_R(n), \quad n'\neq n,
\end{equation*}
where $f_R(n)$ is the propensity function of reaction $R$ and $w_R$ is its reaction vector. Moreover, the total rate of leaving microstate $n$ is given by
\begin{equation}\label{eq:q(n)}
q(n)=\sum_{n'\neq n}q_{n,n'}=\sum_{j=1}^{r_0}k_{0j}+\sum_{i=1}^N\sum_{j=1}^{r_i}k_{ij}n_i,
\end{equation}
which is an affine function of $n$. As usual, we set $q_{n,n}=-q(n)$.

Let $p_t(n) = \mathbb{P}(X_t=n\mid X_0=n_0)$ denote the probability that the system is in microstate $n$ at time $t$, given the initial state $n_0$. Let $p_t=(p_t(n))_{n\in E}$ denote the probability distribution (as a row vector) of the system at time $t$. Then the time evolution of $p_t$ is governed by the chemical master equation
\begin{equation}\label{eq:cme}
\dot{p}_t(n) = \sum_{n'\neq n}\big[p_t(n')q_{n',n}-p_t(n)q_{n,n'}\big], \quad n\in E.
\end{equation}
Equivalently, the chemical master equation can be written in vector form as $\dot{p}_t=p_tQ$, where $Q=(q_{n,n'})_{n,n'\in E}$ denotes the generator matrix of the Markov chain.
A probability distribution $\pi=(\pi(n))_{n\in E}$ is referred to as a \emph{stationary distribution} (or \emph{invariant distribution}) of the system if it satisfies the steady-state balance equation $\pi Q=0$, that is,
\begin{equation}\label{eq:stationary}
\sum_{n'\neq n}\pi(n')q_{n',n} = \sum_{n'\neq n}\pi(n)q_{n,n'}, \quad n\in E.
\end{equation}

We next recall the following important concept \cite{anderson2015stochastic}.

\begin{definition}\label{def:conservation_law}
A nonzero vector $\alpha\in\mathbb{R}^N$ is called a \emph{conservation law} of the linear reaction network if it is orthogonal to all reaction vectors, that is,
\begin{equation*}
\begin{gathered}
\alpha\cdot\nu_{0j}=0,\quad j=1,\cdots,r_0,\\
\alpha\cdot(\nu_{ij}-e_i)=0,\quad i=1,\cdots,N,\;j=1,\cdots,r_i.
\end{gathered}
\end{equation*}
\end{definition}

Clearly, a nonzero vector $\alpha\in\mathbb{R}^N$ is a conservation law if and only if it is orthogonal to every vector in the stoichiometric subspace $\mathcal{P}$. Hence, the number of linearly independent conservation laws is given by
\begin{equation*}
\dim(\mathcal{P}^\perp)=N-\dim(\mathcal{P}).
\end{equation*}
In particular, if the linear reaction network admits no conservation laws, then $\mathcal{P}^\perp=\{0\}$ and hence $\mathcal{P}=\mathbb{R}^N$.

\subsection{Deterministic dynamics of linear reaction networks}\label{subsec:deterministic}
When the molecule numbers of all chemical species are large, the reaction kinetics are often described by deterministic differential equations. The deterministic mass-action kinetics of a linear reaction network is governed by the ordinary differential equation \cite{anderson2015stochastic}
\begin{equation}\label{detmodel}
\dot{z}(t) = \sum_{j=1}^{r_0}k_{0j}\nu_{0j} + \sum_{\ell=1}^N\sum_{j=1}^{r_\ell}k_{\ell j}z_\ell(t)(\nu_{\ell j}-e_\ell),
\end{equation}
where $z(t)=(z_1(t),\cdots,z_N(t))^{\mathsf T}\in\mathbb{R}_{\ge0}^N$, with $z_i(t)$ denoting the concentration of species $S_i$ at time $t$. For simplicity, we assume that the volume of the reaction vessel is fixed at $V=1$, so that the concentration coincides with the molecule number and there is no need to distinguish between them. Here $k_{0j}$ represents the flow (i.e. the number of reaction events occurring per unit time) of the zero-order reaction $R_{0j}$ and $\nu_{0j}$ is the associated reaction vector; $k_{\ell j}z_\ell(t)$ represents the flow of the first-order reaction $R_{\ell j}$ and $\nu_{\ell j}-e_\ell$ is the associated reaction vector. Note that the right-hand side of \eqref{detmodel} is an affine function of $z$, and hence \eqref{detmodel} can be rewritten in vector form as
\begin{equation}\label{eq:affine}
\dot{z}(t) = Az(t)+b,
\end{equation}
where $A=(a_{i\ell})\in\mathbb{R}^{N\times N}$ is a matrix and $b=(b_i)\in\mathbb{R}^N$ is a column vector. Explicitly, the entries of $A$ and $b$ are given by
\begin{equation}\label{eq:jacobian}
a_{i\ell} = \sum_{j=1}^{r_\ell}k_{\ell j}(\nu^{\ell j}_i-\delta_{i\ell}), \quad
b_i = \sum_{j=1}^{r_0}k_{0j}\nu^{0j}_i,
\end{equation}
where $\delta_{i\ell}$ denotes the Kronecker delta. It is a classical result \cite{horn1972general} that the trajectory of the deterministic system \eqref{eq:affine} with initial condition $z(0)\in\mathbb{R}_{\geq0}^N$ remains in the stoichiometric compatibility class $\mathcal{P}_{z(0)}$.

\begin{definition}
For a linear reaction network, the matrix $A$ in \eqref{eq:affine} is called the \emph{first-order influx matrix}, and the vector $b$ in \eqref{eq:affine} is called the \emph{zero-order influx vector}.
\end{definition}

For a linear reaction network, the mean $\mathbb{E}[X_t]$ of the stochastic dynamics satisfies exactly the same differential equation as the deterministic system \eqref{detmodel}. Indeed, multiplying the chemical master equation \eqref{eq:cme} by $n_i$ and summing over all microstates yields
\begin{equation}\label{eq:mean}
\frac{d}{dt}\mathbb{E}[X_t^i]
= \sum_{j=1}^{r_0}k_{0j}\nu^{0j}_i
+\sum_{\ell=1}^N \sum_{j=1}^{r_\ell}k_{\ell j}\mathbb{E}[X_t^{\ell}](\nu^{\ell j}_i-\delta_{i\ell}),
\end{equation}
which is identical to \eqref{detmodel}. Therefore, if the initial mean of the stochastic dynamics is chosen to coincide with the deterministic initial condition, i.e. $\mathbb{E}[X_0]=z(0)$, then $\mathbb{E}[X_t]=z(t)$ for all $t\ge0$. This fact will be used repeatedly throughout the following analysis.

\subsection{Long-time behavior of the deterministic dynamics}
Recall that a real-valued square matrix is called \emph{Metzler} if all its off-diagonal entries are non-negative \cite{berman1994nonnegative}. Note that the first-order influx matrix $A=(a_{i\ell})\in\mathbb{R}^{N\times N}$ in \eqref{eq:affine} is Metzler, because for any $i\neq \ell$,
\begin{equation*}
a_{i\ell}=\sum_{j=1}^{r_\ell}k_{\ell j}\nu^{\ell j}_i\ge0.
\end{equation*}
Let $\sigma(A)$ denote the spectrum of $A$, i.e. the set of all eigenvalues of $A$. By the Perron-Frobenius theorem for Metzler matrices \cite{berman1994nonnegative}, the maximal eigenvalue
\begin{equation*}
\lambda_{\max} := \max\{\operatorname{Re}(\lambda):\lambda\in\sigma(A)\}
\end{equation*}
of $A$ is a real number, and there exists a corresponding eigenvector with non-negative components.

\begin{definition}
An eigenvalue of $A$ is called \emph{semisimple} if it has equal algebraic and geometric multiplicities, or equivalently, if $A$ has no nontrivial Jordan block associated with this eigenvalue. If both the algebraic and geometric multiplicities are equal to $1$, then the eigenvalue is called \emph{simple}.
\end{definition}

We now analyze the long-time behavior of the deterministic dynamics \eqref{eq:affine}, which is determined by (i) the maximal eigenvalue $\lambda_{\max}$ of the first-order influx matrix $A$ and (ii) the position of the zero-order influx vector $b$ relative to the left nullspace of $A$. The main results are summarized in the following proposition, whose proof can be found in \cite{hirsch1974differential}.

\begin{proposition}\label{prop:det_longtime}
The deterministic dynamics \eqref{eq:affine} exhibits the following long-time behavior.
\begin{itemize}
\item[(i)] \textbf{Stable case: $\lambda_{\max}<0$.} The system has a unique globally attracting equilibrium $z^*=-A^{-1}b\ge0$. For any initial condition $z(0)\ge0$, the deterministic trajectory converges exponentially fast to $z^*$, with the convergence rate controlled by $-\lambda_{\max} > 0$. Specifically, for every $\varepsilon\in(0,-\lambda_{\max})$, there exists a constant $C(z(0),\varepsilon)>0$ such that
    \begin{equation*}
    \|z(t)-z^*\|_1\le C(z(0),\varepsilon)e^{-(-\lambda_{\max}-\varepsilon)t}, \quad \text{for all }t\ge0.
    \end{equation*}

\item[(ii)] \textbf{Critical case: $\lambda_{\max}=0$.} Suppose that the zero eigenvalue of $A$ is semisimple and that $b$ is orthogonal to the left nullspace of $A$. For any initial condition $z(0)\ge0$, the deterministic trajectory converges to an equilibrium $z^*\in\mathcal{P}_{z(0)}$ that may depend on the initial condition. The convergence is exponentially fast, with the convergence rate controlled by the spectral gap of $A$, defined by
    \begin{equation}\label{gap}
    \lambda_{\mathrm{gap}}=\inf\{|\operatorname{Re}(\lambda)|:\lambda\in\sigma(A),\;\lambda\neq0\}>0,
    \quad \inf\varnothing:=\infty.
    \end{equation}
    Specifically, for every $\varepsilon\in(0,\lambda_{\mathrm{gap}})$, there exists a constant
    $C(z(0),\varepsilon)>0$ such that
    \begin{equation*}
    \|z(t)-z^*\|_1\le C(z(0),\varepsilon)e^{-(\lambda_{\mathrm{gap}}-\varepsilon)t}, \quad \text{for all }t\ge0.
    \end{equation*}
    \emph{Throughout the paper, when $\lambda_{\rm gap}=\infty$, statements involving an exponential rate $\lambda_{\rm gap}-\varepsilon$ are understood as holding with any prescribed finite rate $\beta>0$.} If the zero eigenvalue of $A$ is not semisimple or if $b$ is not orthogonal to the left nullspace of $A$, then for any initial condition $z(0)>0$, the deterministic trajectory diverges to infinity polynomially fast, i.e. $\lim_{t\to\infty}\|z(t)\|_1=\infty$.

\item[(iii)] \textbf{Unstable case: $\lambda_{\max}>0$.} For any initial condition $z(0)>0$, the deterministic trajectory diverges to infinity exponentially fast, i.e. $\lim_{t\to\infty}\|z(t)\|_1=\infty$.
\end{itemize}
\end{proposition}

\begin{remark}
We emphasize that in the stable and critical regimes, the arbitrarily small loss $\varepsilon$ in the exponential convergence rate is due to the possible presence of Jordan blocks associated with the dominant eigenvalue (in the stable case, the dominant eigenvalue is the maximal eigenvalue $\lambda_{\max}<0$; in the critical case, it is the nonzero eigenvalue with the largest real part) \cite{hirsch1974differential}. In the divergent regimes (polynomial or exponential divergence), we additionally assume that $z(0)>0$ in order to avoid the technical complications that arise when some species are initially absent.
\end{remark}

According to the above proposition, the deterministic system has a unique globally attracting equilibrium $z^*\geq0$ in the stable regime. In the critical regime, however, the system may have multiple equilibria, even within a single stoichiometric compatibility class. For example, for the reaction network $S_1\to S_2$ and $S_1\to S_3$, every $z^*\geq0$ satisfying $z_1^*=0$ is an equilibrium. Hence, the system has infinitely many equilibria within each stoichiometric compatibility class.

The following proposition identifies a set of conserved quantities in the critical regime.

\begin{proposition}\label{prop:critical_affine}
In the critical case of $\lambda_{\max}=0$, suppose that the zero eigenvalue of $A$ is semisimple and $b$ is orthogonal to the left nullspace of $A$. Then there exists a basis $\{\alpha^{(1)},\cdots,\alpha^{(m)}\}$ of the left nullspace of $A$ consisting of non-negative vectors, i.e. $\alpha^{(1)},\cdots,\alpha^{(m)}\in\mathbb{R}_{\geq0}^N$. Moreover, for any left null vector $\alpha$ of $A$, the quantity $\alpha\cdot z(t)$ is constant in time.
\end{proposition}

\begin{proof}
The existence of a basis $\{\alpha^{(1)},\cdots,\alpha^{(m)}\}$ of the left nullspace of $A$ consisting of non-negative vectors follows from the Perron-Frobenius theory for Metzler matrices \cite{berman1994nonnegative,hartwig1990analgebraic}. For any left null vector $\alpha$ of $A$, differentiating $\alpha\cdot z(t)$ along the deterministic dynamics $\dot{z}=Az+b$ yields
\begin{equation*}
\frac{d}{dt}(\alpha\cdot z)=\alpha\cdot(Az+b)=0,
\end{equation*}
where $\alpha\cdot b=0$ follows from the orthogonality of $b$ to the left nullspace of $A$. Therefore, $\alpha\cdot z(t)$ is constant in time.
\end{proof}

Recall from Definition~\ref{def:conservation_law} that a conservation law is a vector $\alpha\in\mathbb{R}^N$ orthogonal to all reaction vectors; it follows immediately that $\alpha\cdot z(t)$ is constant in time. Under the assumptions of Proposition~\ref{prop:critical_affine}, the same invariance holds for every left null vector $\alpha$ of $A$. This raises the following question: in the critical case, is every left null vector of $A$ necessarily a conservation law? The following two examples show that the answer is negative in general.

\begin{example}\label{ex:conservation_examples}
Consider the linear reaction network
\begin{equation*}
S_1 \xrightarrow{k_1} S_2,\quad S_2 \xrightarrow{k_2} S_1.
\end{equation*}
For this system, the first-order influx matrix $A$ and the zero-order influx vector $b$ are given by
\begin{equation*}
A = \begin{pmatrix} -k_1 & k_2 \\ k_1 & -k_2 \end{pmatrix},\quad
b = \begin{pmatrix} 0 \\ 0 \end{pmatrix}.
\end{equation*}
It is easy to verify that the maximal eigenvalue of $A$ is zero and that the zero eigenvalue is semisimple. The left nullspace of $A$ is one-dimensional and is spanned by the nonnegative vector $\alpha=(1,1)$. The reaction vectors are $w_1=(-1,1)$ and $w_2=(1,-1)$. Direct computations show that $\alpha$ is orthogonal to both reaction vectors. Hence, every left null vector of $A$ is also a conservation law.
\end{example}

\begin{example}\label{ex:non_conservation}
Consider the linear reaction network
\begin{equation*}
S_1 \xrightarrow{k_1} S_2,\quad S_1 \xrightarrow{k_2} S_3.
\end{equation*}
For this system, the first-order influx matrix $A$ and the zero-order influx vector $b$ are given by
\begin{equation*}
A = \begin{pmatrix}
-(k_1+k_2) & 0 & 0 \\
k_1 & 0 & 0 \\
k_2 & 0 & 0
\end{pmatrix},\quad
b = \begin{pmatrix} 0 \\ 0 \\ 0 \end{pmatrix}.
\end{equation*}
It is easy to verify that the maximal eigenvalue of $A$ is zero and that the zero eigenvalue is semisimple. The left nullspace of $A$ is two-dimensional and is spanned by the nonnegative vectors
\begin{equation*}
\alpha^{(1)}=(k_1,k_1+k_2,0), \quad \alpha^{(2)}=(k_2,0,k_1+k_2).
\end{equation*}
The reaction vectors are $w_1=(-1,1,0)$ and $w_2=(-1,0,1)$. Direct computations show that neither $\alpha^{(1)}$ nor $\alpha^{(2)}$ is orthogonal to the reaction vectors. Hence, this example shows that a left null vector of $A$ is not necessarily a conservation law.
\end{example}

\section{Long-time behavior of the stochastic dynamics}\label{sec:main}
We now state the main results of this paper. We first prove that the stochastic dynamics of linear reaction networks are always non-explosive. We then investigate the long-time behavior of \emph{reducible} stochastic linear reaction networks in the stable ($\lambda_{\max}<0$), critical ($\lambda_{\max}=0$), and unstable ($\lambda_{\max}>0$) regimes, respectively.

\subsection{Non-explosivity}
We first establish a basic fact that will be used throughout the paper: the stochastic dynamics of a linear reaction network cannot explode in finite time. Let $J_0=0$ and, as long as $J_k<\infty$, define the successive jump times of the process $X_t$ by
\begin{equation*}
J_{k+1}:=\inf\{t\ge J_k:X_t\ne X_{J_k}\}, \quad \inf\varnothing:=\infty;
\end{equation*}
once $J_k=\infty$ (this occurs when the system reaches an absorbing state), we set $J_{k+1}=\infty$. Intuitively, $J_k$ is the $k$th jump time of the process. Before reaching an absorbing state, the jump times are finite and strictly increasing; after reaching an absorbing state, the system can no longer make any jumps, and hence the subsequent jump times are all equal to $\infty$. For example, if the reaction network consists of the single reaction $S\to\varnothing$ and there are initially $n_0$ molecules of species $S$, then $J_1<\cdots<J_{n_0}<\infty$ and $J_k=\infty$ for every $k\ge n_0+1$, since the state $n=0$ is absorbing.

Under the above notation, the explosion time of the process $X_t$ is defined by
\begin{equation*}
J_\infty:=\lim_{k\to\infty}J_k.
\end{equation*}
Intuitively, $J_\infty<\infty$ if and only if infinitely many jumps occur within a finite time interval. Given an initial state $n_0$, the process is called \emph{explosive} if $\mathbb{P}_{n_0}(J_\infty<\infty)>0$, and \emph{non-explosive} if $\mathbb{P}_{n_0}(J_\infty=\infty)=1$
\cite{norris1998markov}.

\begin{theorem}\label{thm:non-explosive}
The stochastic dynamics of a linear reaction network is non-explosive --- for every initial state $n_0\in E$, we have
\begin{equation*}
\mathbb{P}_{n_0}(J_\infty=\infty)=1.
\end{equation*}
\end{theorem}

\begin{proof}
The proof is given in Supplementary Section~S1.
\end{proof}

\subsection{Long-time behavior in the stable case}
We next investigate the long-time behavior of a \emph{reducible} stochastic linear reaction network, i.e. the convergence of the time-dependent distribution $p_t$ of the system to a stationary limiting distribution $\pi$ as $t\to\infty$, together with the exponential rate at which this convergence occurs. Note that both $p_t=(p_t(n))_{n\in E}$ and $\pi=(\pi(n))_{n\in E}$ are probability measures on the state space $E=\mathbb{N}^N$. We quantify the distance between probability measures using two classical metrics \cite{gibbs2002choosing,villani2009optimal}.

\begin{definition}
Let $\mu$ and $\nu$ be two probability measures on $E$ with finite first moments. Recall that the \emph{total variation distance} between $\mu$ and $\nu$ is defined by
\begin{equation*}
\|\mu-\nu\|_{\mathrm{TV}}
= \frac12\sum_{n\in E} |\mu(n)-\nu(n)|
= \sup_{B\subseteq E} |\mu(B)-\nu(B)|,
\end{equation*}
and the \emph{$L^1$-Wasserstein distance} between $\mu$ and $\nu$ is defined by
\begin{equation}\label{eq:W1_def}
W_1(\mu,\nu)=\inf_{\gamma\in\Pi(\mu,\nu)}\sum_{x,y\in E}\|x-y\|_1\gamma(x,y),
\end{equation}
where $\Pi(\mu,\nu)$ denotes the set of all couplings of $\mu$ and $\nu$, that is, all probability measures on $E\times E$ with $\mu$ and $\nu$ as their marginal distributions.
\end{definition}

By the Kantorovich-Rubinstein duality theorem, we have
\begin{equation}\label{eq:KR_duality}
W_1(\mu,\nu)=\sup_{\substack{f:E\to\mathbb{R}\\ |f(x)-f(y)|\le\|x-y\|_1}}
\left|\sum_{x\in E}f(x)\mu(x)-\sum_{x\in E}f(x)\nu(x)\right|.
\end{equation}
On the discrete state space $E$ equipped with the $\ell^1$-norm, any two distinct states are at distance at least $1$. Consequently, the total variation distance is bounded above by the Wasserstein distance \cite{villani2009optimal}, that is,
\begin{equation}\label{eq:TV_vs_W1}
\|\mu-\nu\|_{\mathrm{TV}}\le W_1(\mu,\nu).
\end{equation}

To make the paper self-contained, we briefly recall some terminology for continuous-time Markov chains on a countable state space \cite{norris1998markov}. A set of states is called \emph{closed} if the process cannot leave it. A closed set is called \emph{irreducible} if any two states in it communicate with each other. A state is called \emph{recurrent} if the process returns to it infinitely often; otherwise, it is called \emph{transient} (it is visited only finitely many times). Equivalently, a state is transient if and only if it does not belong to any closed recurrent class. A recurrent state is called \emph{positive recurrent} if its expected return time is finite; otherwise, it is called \emph{null recurrent} (its expected return time is infinite). An irreducible closed class is positive recurrent if and only if it supports a unique stationary distribution.

We first consider the stable regime of $\lambda_{\max}<0$. In this case, the deterministic trajectory converges to the unique globally attracting equilibrium $z^*=-A^{-1}b\ge0$ exponentially fast (Proposition~\ref{prop:det_longtime}). Previous studies \cite{gupta2014scalable,cappelletti2025stability,xu2026exponential} have established exponential ergodicity for stochastic linear reaction networks under additional irreducibility or structural assumptions. Here we give a structural characterization of the long-time behavior of a reducible stochastic linear reaction network without imposing any additional assumptions.

In the stable case, some species may have zero equilibrium concentrations, i.e. $z_i^*=0$ for some species $S_i$. Such a situation significantly complicates the analysis. We therefore classify all species according to the support of the equilibrium $z^*$:
\begin{equation}\label{eq:def_I_plus}
I_0=\{i:z_i^*=0\},\quad
I_+=\{i:z_i^*>0\}.
\end{equation}
Here $I_0$ indexes the species that become extinct under the deterministic dynamics, while $I_+$ indexes the species that persist under the deterministic dynamics. This partition of the species determines the following absorbing set for the stochastic dynamics:
\begin{equation}\label{eq:def_R}
\mathcal{B}=\{n\in E:\;n_i=0\text{ for all }i\in I_0\}.
\end{equation}
Intuitively, the set $\mathcal{B}$ consists of all microstates where all deterministically extinct species have zero copy numbers. The following theorem shows that, in the stable regime, a stochastic linear reaction network must enter the absorbing set $\mathcal{B}$ in finite time. In other words, the deterministically extinct species must also become extinct in the stochastic sense.

\begin{theorem}\label{thm:hitting_exp}
Assume that $\lambda_{\max}<0$. Then $\mathcal{B}$ is the unique closed irreducible class of the process and is also positive recurrent. Moreover, for any $\varepsilon\in(0,-\lambda_{\max})$ and initial state $n_0\in E$,
\begin{equation*}
\mathbb{E}_{n_0}\bigl[e^{(-\lambda_{\max}-\varepsilon)\tau_{\mathcal{B}}}\bigr]<\infty,
\end{equation*}
where $\tau_{\mathcal{B}}=\inf\{t\ge0:X_t\in\mathcal{B}\}$ is the hitting time of $\mathcal{B}$.
\end{theorem}

\begin{proof}
The proof is given in Supplementary Section~S2.
\end{proof}

The following theorem shows that, in the stable regime, a stochastic linear reaction network is exponentially ergodic on the absorbing set $\mathcal{B}$ under the $L^1$-Wasserstein distance.

\begin{theorem}\label{thm:stable_case}
Assume that $\lambda_{\max}<0$. Let $\pi=(\pi(n))_{n\in E}$ be the unique stationary distribution supported on the positive recurrent class $\mathcal{B}$.
\begin{itemize}
\item[(i)] For every $\varepsilon\in(0,-\lambda_{\max})$ and initial state $n_0\in E$, there exists a constant $C(n_0,\varepsilon)<\infty$ such that
    \begin{equation}\label{eq:exponential_ergodic_stable_case}
    W_1\bigl(p_t,\pi\bigr)\le C(n_0,\varepsilon)\,e^{-(-\lambda_{\max}-\varepsilon)t},\quad\text{for all }t\ge0,
    \end{equation}
    where $W_1$ denotes the $L^1$-Wasserstein distance defined in \eqref{eq:W1_def}.
    \item[(ii)] There exist an initial state $n_0\in E$ and a constant $c>0$ such that for all sufficiently large $t$,
    \begin{equation}\label{optimalstable}
    W_1\bigl(p_t,\pi\bigr)\ge ce^{-(-\lambda_{\max})t}.
    \end{equation}
\end{itemize}
\end{theorem}

\begin{proof}
The proof is given in Section~\ref{sec:stable_proof} using coupling techniques.
\end{proof}

Part (i) of the above theorem shows that, in the stable regime, the time-dependent distribution $p_t$ converges to the stationary distribution $\pi$ at an exponential rate bounded above by $-\lambda_{\max}$, while part (ii) shows that $-\lambda_{\max}$ is the optimal exponential convergence rate.

The following corollary shows that, in the stable regime, a stochastic linear reaction network is also exponentially ergodic on the absorbing set $\mathcal{B}$ under the total variation distance.

\begin{corollary}\label{cor:TV_stable}
Assume that $\lambda_{\max}<0$. Then all conclusions remain valid when the Wasserstein distance $W_1$ is replaced by the total variation distance $\|\cdot\|_{\mathrm{TV}}$. Moreover, for the total variation distance, $-\lambda_{\max}$ is still the optimal exponential convergence rate in the sense of \eqref{optimalstable}.
\end{corollary}

\begin{proof}
Since the total variation distance is bounded above by the $L^1$-Wasserstein distance, for every $\varepsilon\in(0,-\lambda_{\max})$ and initial state $n_0\in E$, there exists a constant $C(n_0,\varepsilon)<\infty$ such that
\begin{equation*}
\|p_t-\pi\|_{\mathrm{TV}} \le W_1\bigl(p_t,\pi\bigr)
\le C(n_0,\varepsilon)\,e^{-(-\lambda_{\max}-\varepsilon)t},\quad\text{for all }t\ge0.
\end{equation*}
The optimal exponential convergence rate under the total variation distance is proved in Supplementary Section~S2.
\end{proof}

\begin{corollary}\label{cor:irred}
Assume that $\lambda_{\max}<0$. Then the following statements are equivalent:
\begin{itemize}
\item[(i)] $I_0=\varnothing$, i.e. all species persist under the deterministic dynamics;
\item[(ii)] the stochastic dynamics is irreducible and positive recurrent on the entire state space $E$.
\end{itemize}
\end{corollary}

\begin{proof}
(i)~$\Rightarrow$~(ii): If $I_0=\varnothing$, then clearly $\mathcal{B}=E$. By Theorem~\ref{thm:hitting_exp}, $\mathcal{B}$ is the unique closed irreducible positive recurrent class. Hence, the stochastic dynamics is irreducible and positive recurrent on the entire state space $E$.

(ii)~$\Rightarrow$~(i): By Theorem~\ref{thm:hitting_exp}, $\mathcal{B}$ is the unique closed irreducible class. If the stochastic dynamics is irreducible and positive recurrent on the entire state space $E$, then we must have $\mathcal{B}=E$. This implies that $I_0=\varnothing$.
\end{proof}

Beyond exponential ergodicity, our results also provide a precise characterization of the support of the stationary distribution, namely, the unique closed irreducible class $\mathcal{B}$ into which the process is eventually absorbed. We next illustrate this result with two examples.

\begin{example}\label{ex:gene_expression}
According to the central dogma of molecular biology, gene expression in a living cell proceeds through two successive steps: transcription, in which mRNA $M$ is synthesized from DNA, and translation, in which protein $P$ is synthesized from mRNA. The classical two-stage model of gene expression \cite{shahrezaei2008analytical} explicitly describes these two steps and is characterized by the following reactions:
\begin{equation*}
\varnothing\xlongrightarrow{k_M}M,\quad M\xlongrightarrow{k_P}M+P,\quad
M\xlongrightarrow{\gamma_M}\varnothing,\quad P\xlongrightarrow{\gamma_P}\varnothing.
\end{equation*}
The first two reactions correspond to transcription and translation, while the last two reactions represent the degradation of the gene products. Here $k_M$ is the transcription rate, $k_P$ is the translation rate, $\gamma_M$ is the degradation rate of mRNA, and $\gamma_P$ is the degradation rate of protein.

We now examine the stochastic dynamics of the two-stage model. In fact, the stationary distribution of mRNA and protein copy numbers for this model can be computed in closed form \cite{bokes2012exact, li2021steady}. The stationary distribution of mRNA copy numbers is Poissonian; however, the stationary distribution of protein copy numbers is considerably more complicated and involves confluent and generalized hypergeometric functions \cite{li2021steady}. We next investigate the exponential ergodicity of the stochastic dynamics. For this reaction network, the first-order influx matrix $A$ and the zero-order influx vector $b$ are given by
\begin{equation*}
A = \begin{pmatrix}
-\gamma_M & 0 \\
k_P & -\gamma_P
\end{pmatrix},\quad
b = \begin{pmatrix} k_M \\ 0 \end{pmatrix}.
\end{equation*}
Because $A$ is lower triangular with strictly negative diagonal entries, its spectrum is given by $\sigma(A)=\{-\gamma_M,-\gamma_P\}$, and hence the maximal eigenvalue of $A$ is $\lambda_{\max}=-\min\{\gamma_M,\gamma_P\}<0$. By Proposition~\ref{prop:det_longtime}, the deterministic system has a unique globally attracting equilibrium $z^*=(z_M^*,z_P^*)=-A^{-1}b$ with
\begin{equation*}
z_M^*=\frac{k_M}{\gamma_M}>0,\quad
z_P^*=\frac{k_Pk_M}{\gamma_M\gamma_P}>0.
\end{equation*}
Clearly, both components of $z^*$ are positive. Hence, by Theorem~\ref{thm:stable_case}, the entire state space $E=\mathbb{N}^2$ is the unique closed irreducible positive recurrent class of the stochastic dynamics, and the system converges exponentially fast to the stationary distribution under both the total variation and $L^1$-Wasserstein distances. The optimal exponential convergence rate is determined by $-\lambda_{\max}=\min\{\gamma_M,\gamma_P\}$, that is, by the smaller of the mRNA degradation rate and the protein degradation rate. In naturally occurring systems, mRNA often has a much shorter lifetime than its protein counterpart, i.e. $\gamma_P\ll\gamma_M$ \cite{shahrezaei2008analytical,jia2017simplification}. In this case, the exponential convergence rate to the stationary distribution is determined by the protein degradation rate $\gamma_P$.
\end{example}

We next consider an example in which some species may become extinct.

\begin{example}\label{ex:open_monomolecular}
Recall that a linear reaction network is called \emph{monomolecular} if every reaction produces at most one molecule, i.e.
\begin{equation*}
\|\nu_{ij}\|_1
=\sum_{\ell=1}^N \nu_\ell^{ij}
\in\{0,1\}
\end{equation*}
for every zero-order or first-order reaction $R_{ij}$ \cite{kelly1979reversibility, gadgil2005stochastic, heuett2006grand, jahnke2007solving}. Clearly, a monomolecular reaction network consists of synthesis reactions $\varnothing\to S_i$, degradation reactions $S_i\to\varnothing$, and conversion reactions $S_i\to S_j$. It is called \emph{open} if it contains at least one synthesis or degradation reaction; otherwise, it is called \emph{closed}. Each monomolecular reaction network is naturally associated with a directed \emph{reaction graph} whose vertices are $\{S_1,\cdots,S_N,\varnothing\}$, that is, all species together with $\varnothing$ whenever synthesis or degradation reactions are present.

The following theorem characterizes when an open monomolecular reaction network is in the stable regime and establishes the exponential ergodicity of its stochastic dynamics.

\begin{theorem}
Consider the stochastic dynamics of an open monomolecular reaction network. If each species $S_i$ can reach $\varnothing$ in the reaction graph, i.e. for every species $S_i$, there exists a directed path
\begin{equation*}
S_i\to S_{i_1}\to S_{i_2}\to\cdots\to S_{i_k}\to\varnothing,
\end{equation*}
then the maximal eigenvalue of the first-order influx matrix $A$ is negative, i.e. $\lambda_{\max}<0$. The system has a unique stationary distribution supported on the absorbing set $\mathcal{B}$, which has the following Poisson-product form:
\begin{equation}\label{monostationary}
\pi(n)=\prod_{i=1}^Ne^{-z_i^*}\frac{(z_i^*)^{n_i}}{n_i!},\quad n\in\mathcal{B},
\end{equation}
where $z^*=-A^{-1}b\ge0$ is the unique globally attracting equilibrium of the deterministic dynamics. Furthermore, the system is exponentially ergodic under both the $L^1$-Wasserstein distance and the total variation distance, with optimal exponential convergence rate $-\lambda_{\max}$, i.e. \eqref{eq:exponential_ergodic_stable_case} and \eqref{optimalstable} hold.
\end{theorem}

\begin{proof}
An open monomolecular reaction network can be written in the general form
\begin{equation*}
\varnothing\xrightarrow{\beta_i}S_i,\quad S_i\xrightarrow{d_i}\varnothing,\quad S_i\xrightarrow{k_{ij}}S_j,\quad i\ne j,
\end{equation*}
where $\beta_i,d_i,k_{ij}\ge0$. The deterministic kinetics of the network is governed by the differential equation
\begin{equation*}
\dot z=Az+b=\sum_{i=1}^N\left[\sum_{j\ne i}k_{ij}(e_j-e_i)-d_ie_i\right]z_i+(\beta_1,\cdots,\beta_N)^{\mathsf T}.
\end{equation*}
Thus, the first-order influx matrix is given by
\begin{equation*}
A=\begin{pmatrix}
-d_1-\sum_{j\ne1}k_{1j} & k_{21} & \cdots & k_{N1}\\
k_{12} & -d_2-\sum_{j\ne2}k_{2j} & \cdots & k_{N2}\\
\vdots & \vdots & \ddots & \vdots\\
k_{1N} & k_{2N} & \cdots & -d_N-\sum_{j\ne N}k_{Nj}
\end{pmatrix},
\end{equation*}
and the zero-order influx vector is given by $b=(\beta_1,\cdots,\beta_N)^{\mathsf T}$.

Set $s_i=\sum_{j\ne i}k_{ij}$. By the Gershgorin circle theorem \cite{horn2012matrix}, every eigenvalue of $A$ lies in one of the discs
\begin{equation*}
D_i=\{\lambda\in\mathbb C:|\lambda+d_i+s_i|\le s_i\},\quad i=1,\cdots,N.
\end{equation*}
Each $D_i$ lies in the closed left half-plane and can intersect the imaginary axis only at $0$. Hence,
\begin{equation*}
\sigma(A)\subseteq\{z\in\mathbb C:\operatorname{Re}z\le0\},\quad
\sigma(A)\cap\mathrm{i}\mathbb{R}\subseteq\{0\}.
\end{equation*}
It remains to show that $0\notin\sigma(A)$.

Suppose, to the contrary, that $A^{\mathsf T}x=0$ for some $x\ne0$, and choose $i$ such that $|x_i|=\|x\|_\infty>0$. Then the $i$th equation of $A^{\mathsf T}x=0$ gives
\begin{equation}\label{temprecursive}
(d_i+s_i)|x_i|
=\bigg|\sum_{j\ne i}k_{ij}x_j\bigg|
\le s_i\|x\|_\infty
=s_i|x_i|.
\end{equation}
This implies that $d_i=0$. Equality must therefore hold throughout \eqref{temprecursive}, which in turn implies
\begin{equation}\label{recursive}
|x_j|=\|x\|_\infty\quad\text{whenever }k_{ij}>0.
\end{equation}
By assumption, there exists a directed path
\begin{equation*}
S_i\to S_{i_1}\to S_{i_2}\to\cdots\to S_{i_k}\to\varnothing
\end{equation*}
in the reaction graph from $S_i$ to $\varnothing$. This implies that $k_{ii_1},k_{i_1i_2},\cdots,k_{i_{k-1}i_k}>0$. Applying \eqref{recursive} successively along this path yields $|x_{i_k}|=\|x\|_\infty$. Applying \eqref{temprecursive} with $i_k$ in place of $i$ then gives $d_{i_k}=0$, contradicting the existence of the degradation reaction $S_{i_k}\to\varnothing$. Thus, $\ker(A^{\mathsf T})=\{0\}$ and hence $0\notin\sigma(A)$. The preceding Gershgorin estimate now implies that every eigenvalue of $A$ has strictly negative real part, and therefore $\lambda_{\max}<0$.

Proposition~\ref{prop:det_longtime} thus shows that $z^*=-A^{-1}b\ge0$ is the unique globally attracting equilibrium of the deterministic system. Moreover, it is easy to see that the set of persistent species is given by
\begin{equation}\label{property}
I_+ = \{i:\text{there exists a directed path from }\varnothing\text{ to }S_i\},
\end{equation}
and Theorem~\ref{thm:stable_case} implies that the unique stationary distribution is supported on the absorbing set
\begin{equation*}
\mathcal{B} = \{n\in E:n_i=0\text{ for all }i\notin I_+\}.
\end{equation*}
To proceed, we recall that a reaction network is called \emph{weakly reversible} if every connected component of its reaction graph is strongly connected \cite{anderson2010product}. Consider the reaction subnetwork induced by $\{\varnothing\}\cup\{S_i:i\in I_+\}$. By using \eqref{property}, $\varnothing$ can reach every species in this subnetwork. Conversely, every such species can reach $\varnothing$ by assumption. Moreover, all intermediate species on these paths also belong to $I_+$, and hence the subnetwork is weakly reversible. Therefore, the stationary distribution of the subnetwork has the following Poisson-product form \cite{jahnke2007solving,anderson2010product}:
\begin{equation*}
\pi(n)=\prod_{i\in I_+}e^{-z_i^*}\frac{(z_i^*)^{n_i}}{n_i!},\quad n\in\mathcal{B}.
\end{equation*}
Since $z_i^*=n_i=0$ for all $i\notin I_+$ and $n\in\mathcal{B}$, the stationary distribution of the original network is given by \eqref{monostationary}. Finally, exponential ergodicity under the total variation and $L^1$-Wasserstein distances, together with the optimal exponential convergence rate $-\lambda_{\max}$, follows from Theorem~\ref{thm:stable_case} and Corollary~\ref{cor:TV_stable}.
\end{proof}
\end{example}

\subsection{Long-time behavior in the critical case}
We now turn to the critical regime of $\lambda_{\max}=0$ and analyze the long-time behavior of the stochastic system for a given initial state $n_0\in E$. Throughout this section, we impose the following two conditions:
\begin{enumerate}
\item[(C1)] the zero eigenvalue of $A$ is semisimple;
\item[(C2)] $b$ is orthogonal to the left nullspace of $A$.
\end{enumerate}
These two conditions were also imposed in Proposition~\ref{prop:det_longtime} to guarantee the exponential convergence of the deterministic system. If the initial condition of the deterministic system is chosen to coincide with the initial state of the stochastic system, i.e. $z(0)=n_0\geq 0$, then the deterministic trajectory converges exponentially fast to an equilibrium $z^*\in\mathcal{P}_{n_0}$, which may depend on $n_0$, and the convergence rate is determined by $\lambda_{\mathrm{gap}}$. Since the limiting equilibrium depends on the initial state $n_0$, the set $I_+$ of persistent species, the set $I_0$ of extinct species, and the absorbing set $\mathcal{B}$ may also depend on $n_0$, as defined in \eqref{eq:def_I_plus} and \eqref{eq:def_R}.

The following result is the counterpart of Theorem~\ref{thm:hitting_exp} for the critical regime. It shows that every deterministically extinct species must also become extinct in the stochastic sense.

\begin{theorem}\label{thm:hitting_critical}
Assume that $\lambda_{\max}=0$ and that conditions (C1) and (C2) are satisfied. Then for every initial state $n_0\in E$, the set $\mathcal{B}$ defined in \eqref{eq:def_R} is closed. Moreover, for every $\varepsilon\in(0,\lambda_{\mathrm{gap}})$,
\begin{equation*}
\mathbb{E}_{n_0}\bigl[e^{(\lambda_{\mathrm{gap}}-\varepsilon)\tau_{\mathcal{B}}}\bigr] < \infty,
\end{equation*}
where $\tau_{\mathcal{B}} = \inf\{t\geq0:X_t\in\mathcal{B}\}$ is the hitting time of $\mathcal{B}$.
\end{theorem}

\begin{proof}
The proof is given in Supplementary Section~S3.
\end{proof}

We next investigate the long-time behavior in the critical regime. As shown in Example~\ref{ex:non_conservation}, a left null vector of the first-order influx matrix $A$ may not correspond to a conservation law. In this case, the long-time behavior of the stochastic system may differ from that of the deterministic system, and an additional condition is needed to exclude such pathological cases.

\begin{definition}\label{def:restriction_subnetwork}
Let $(\mathcal{S},\mathcal{C},\mathcal{R})$ be a general reaction network, and let $\mathcal{S}'\subseteq\mathcal{S}$ be a subset of species. For a reaction $R\in\mathcal{R}$ with source complex $\mu_R=(\mu_i^R)_{S_i\in\mathcal{S}}$ and product complex $\nu_R=(\nu_i^R)_{S_i\in\mathcal{S}}$, its \emph{restriction} to $\mathcal{S}'$ is defined as the formal reaction
\begin{equation*}
R|_{\mathcal{S}'}:\sum_{S_i\in\mathcal{S}'}\mu_i^RS_i\longrightarrow\sum_{S_i\in\mathcal{S}'}\nu_i^RS_i.
\end{equation*}
For the reaction network $(\mathcal{S},\mathcal{C},\mathcal{R})$, its \emph{restriction} to $\mathcal{S}'$ is defined as the reaction subnetwork $(\mathcal{S}',\mathcal{C}',\mathcal{R}')$, where
\begin{equation*}
\begin{gathered}
\mathcal{R}'=\left\{R|_{\mathcal{S}'}:\mu_R|_{\mathcal{S}\setminus\mathcal{S}'}=0,\ \mu_R|_{\mathcal{S}'}\neq\nu_R|_{\mathcal{S}'}\right\},\quad
\mathcal{C}'=\left\{\mu_R|_{\mathcal{S}'},\nu_R|_{\mathcal{S}'}:R|_{\mathcal{S}'}\in\mathcal{R}'\right\}.
\end{gathered}
\end{equation*}
Here the condition $\mu_R|_{\mathcal{S}\setminus\mathcal{S}'}=0$ means that, in the restricted network, we only retain those reactions $R\in\mathcal{R}$ whose reactant complexes involve species exclusively from $\mathcal{S}'$.
\end{definition}

Let $I_+$ be the set of persistent species and let $I_0$ be the set of extinct species as defined in \eqref{eq:def_I_plus}. Once the process enters the absorbing set $\mathcal{B}$, first-order reactions with reactants in $I_0$ have zero propensities and are no longer dynamically relevant. Therefore, it is natural to restrict the linear reaction network to $I_+$ according to Definition~\ref{def:restriction_subnetwork}. For each restricted reaction $R|_{I_+}\in\mathcal{R}'$, we denote its reaction vector by $w_R|_{I_+} := \nu_R|_{I_+}-\mu_R|_{I_+}$. To distinguish between qualitatively different long-time behaviors in the critical regime, we introduce the following definition.

\begin{definition}\label{regular}
In the critical regime of $\lambda_{\max}=0$, we say that the \emph{regularity condition} is satisfied if every left null vector of $A$ induces a conservation law for the reaction network restricted to $I_+$, i.e. if
\begin{equation*}
\alpha|_{I_+}\cdot w_R|_{I_+}=0,\quad \text{for all }\alpha\in\ker(A^{\mathsf T})\text{ and }R|_{I_+}\in\mathcal{R}',
\end{equation*}
where $\alpha|_{I_+}$ denotes the restriction of $\alpha$ to $I_+$.
\end{definition}

Assume that the regularity condition holds. By Proposition \ref{prop:critical_affine}, there exists a basis $\{\alpha^{(1)},\cdots,\alpha^{(m)}\}$ of the left nullspace of $A$ consisting of non-negative vectors. Since all microstates in the absorbing set $\mathcal{B}$ have zero components for $i\in I_0$, once the process $X_t$ has entered $\mathcal{B}$ (i.e. after the hitting time $\tau_\mathcal{B}$), the quantities $\alpha^{(i)}\cdot X_t$ remain constant in time. Before the process enters $\mathcal{B}$, however, these quantities $\alpha^{(i)}\cdot X_t$ may vary, since some left null vectors of $A$ may not correspond to conservation laws of the original reaction network. For any $c = (c_1,\cdots,c_m)\in\mathbb{R}_{\ge0}^m$, we define
\begin{equation}\label{Ac}
\mathcal{B}_{c} = \{n\in\mathcal{B}:\alpha^{(i)}\cdot n=c_i,\;i=1,\cdots,m\}.
\end{equation}
Clearly, $\mathcal{B}$ is the disjoint union of all $\mathcal{B}_{c}$, i.e. $\mathcal{B} = \bigcup_{c\in\mathbb{R}_{\ge0}^m}\mathcal{B}_{c}$. For the given initial state $n_0\in E$, the process $X_t$ is absorbed into $\mathcal{B}_{c}$ with probability
\begin{equation*}
p_{n_0}(c) := \mathbb{P}_{n_0}\big(\alpha^{(i)}\cdot X_{\tau_{\mathcal{B}}}=c_i,\;i=1,\cdots,m\big),
\end{equation*}
where $\tau_{\mathcal{B}}$ denotes the hitting time of $\mathcal{B}$.

The following result is the counterpart of Theorem \ref{thm:stable_case} for the critical regime. It characterizes the exponential convergence of the system to a stationary limiting distribution under both the total variation and $L^1$-Wasserstein distances.

\begin{theorem}\label{thm:critical_regular}
Fix an initial state $n_0\in E$. Assume that $\lambda_{\max}=0$, that conditions (C1) and (C2) are satisfied, and that the regularity condition in Definition~\ref{regular} is satisfied. Then there exists a unique stationary distribution $\pi_c$ within each nonempty $\mathcal{B}_c$. Moreover, for every $\varepsilon\in(0,\lambda_{\mathrm{gap}})$, there exists $C(n_0,\varepsilon)<\infty$ such that
\begin{equation}\label{criticalestimate}
\|p_t-\pi_{n_0}\|_{\mathrm{TV}}
\leq W_1\bigl(p_t,\pi_{n_0}\bigr)
\leq C(n_0,\varepsilon)e^{-(\lambda_{\mathrm{gap}}-\varepsilon)t},
\quad t\geq0,
\end{equation}
where $\lambda_{\mathrm{gap}}$ is the spectral gap of $A$ defined in \eqref{gap}, and
\begin{equation*}
\pi_{n_0} = \sum_{\substack{c\in\mathbb R_{\geq0}^m:\mathcal{B}_c\neq\varnothing}}p_{n_0}(c)\pi_c
\end{equation*}
is the stationary limiting distribution of the system.
\end{theorem}

\begin{proof}
The proof is given in Section~\ref{sec:critical_regular_proof} using coupling techniques.
\end{proof}


\begin{remark}\label{rem:basis_independence}
Note that the definition of $\mathcal{B}_c$ in \eqref{Ac} depends on the choice of the non-negative basis $\alpha = \{\alpha^{(1)},\cdots,\alpha^{(m)}\}$ of the left nullspace of $A$. If another non-negative basis $\beta = \{\beta^{(1)},\cdots,\beta^{(m)}\}$ is chosen, let $O = (O_{ij})$ be the invertible matrix representing the change of basis, i.e.
\begin{equation*}
\beta^{(i)} = \sum_{j=1}^m O_{ij}\,\alpha^{(j)}.
\end{equation*}
Then it is straightforward to verify that
\begin{equation*}
\mathcal{B}_c^{(\alpha)} = \mathcal{B}_{O c}^{(\beta)}.
\end{equation*}
Therefore, the partition $\mathcal{B} = \bigcup_{c}\mathcal{B}_{c}$ of the absorbing set $\mathcal{B}$ is independent of the choice of the non-negative basis. In particular, each subset $\mathcal{B}_c^{(\alpha)}$ defined under the basis $\alpha$ corresponds exactly to the subset $\mathcal{B}_{O c}^{(\beta)}$ defined under the basis $\beta$.
\end{remark}

\begin{theorem}\label{thm:critical_sharpness}
Assume that $\lambda_{\max}=0$, that conditions (C1) and (C2) are satisfied, and that $\lambda_{\rm gap}<\infty$. Then there exist an initial state $n_0\in E$ and a constant $c>0$ such that, for all sufficiently large $t$,
\begin{equation}\label{optimalcritical}
W_1\bigl(p_t,\pi_{n_0}\bigr) \geq \bigl\|p_t-\pi_{n_0}\bigr\|_{\mathrm{TV}}
\geq c e^{-\lambda_{\mathrm{gap}}t},
\end{equation}
where $\pi_{n_0}$ is the stationary limiting distribution of the system.
\end{theorem}

\begin{proof}
The proof is given in Supplementary Section~S3.
\end{proof}

Recall that, in the stable regime, we have shown that $-\lambda_{\max}$ is the optimal exponential convergence rate. In the critical regime, when the regularity condition holds for every initial state, Theorems~\ref{thm:critical_regular} and \ref{thm:critical_sharpness} together show that $\lambda_{\rm gap}$ is the optimal uniform exponential convergence rate. If the regularity condition fails for some initial state, exponential convergence fails for that state, as shown below in Theorem~\ref{thm:critical_irregular}.

The following corollary shows that if every left null vector of $A$ is a conservation law for the original (rather than restricted) reaction network, then the regularity condition in Definition \ref{regular} is automatically satisfied, and hence exponential convergence holds.

\begin{corollary}\label{cor:fully_conservative}
Assume that $\lambda_{\max}=0$ and that conditions (C1) and (C2) are satisfied. If every left null vector of $A$ is a conservation law for the original reaction network, then the regularity condition in Definition~\ref{regular} is automatically satisfied. Moreover, for every initial state $n_0\in E$, the time-dependent distribution $p_t$ converges exponentially fast to a stationary limiting distribution $\pi_{n_0}$ under both the $L^1$-Wasserstein and total variation distances, as in Eq.~\eqref{criticalestimate}, and $\lambda_{\rm gap}$ is the optimal uniform exponential convergence rate in the sense of Eq.~\eqref{optimalcritical}.
\end{corollary}

\begin{proof}
Fix an initial state $n_0\in E$ and restrict the linear reaction network to $I_+$. Let $R|_{I_+}\in\mathcal{R}'$ be a reaction in the restricted network. By Definition~\ref{def:restriction_subnetwork}, the restricted network retains only those reactions $R\in\mathcal{R}$ whose reactant complexes involve species exclusively from $I_+$. Thus, $R$ is either a zero-order reaction $R_{0j}$ or a first-order reaction $R_{ij}$ for some $i\in I_+$. In either case, its source complex contains no species in $I_0$, which implies that $\mu_R|_{I_0}=0$.

Moreover, Theorem~\ref{thm:hitting_critical} shows that the absorbing set $\mathcal{B}$ defined in Eq.~\eqref{eq:def_R} is closed. Clearly, the zero vector $\mathbf{0}\in\mathcal{B}$ and the standard basis vector $e_i\in\mathcal{B}$ whenever $i\in I_+$. If $R=R_{0j}$ is a zero-order reaction, then it induces a transition from microstate $\mathbf{0}$ to microstate $\nu_R=\nu_{0j}$ with rate $k_{0j}$; if $R=R_{ij}$ is a first-order reaction for some $i\in I_+$, then it induces a transition from microstate $e_i$ to microstate $\nu_R=\nu_{ij}$ with rate $k_{ij}$. In either case, we have $\nu_R\in\mathcal{B}$ since $\mathcal{B}$ is closed. Therefore, $\nu_R|_{I_0}=0$ and hence
\begin{equation*}
w_R|_{I_0} = \nu_R|_{I_0}-\mu_R|_{I_0} = 0.
\end{equation*}

Since every left null vector $\alpha$ of $A$ is a conservation law for the original reaction network, we have $\alpha\cdot w_R=0$ for each reaction $R$. Consequently,
\begin{equation*}
\alpha|_{I_+}\cdot w_R|_{I_+}
= \alpha\cdot w_R-\alpha|_{I_0}\cdot w_R|_{I_0}
= \alpha\cdot w_R
= 0.
\end{equation*}
Hence, the regularity condition in Definition~\ref{regular} is satisfied for the given initial state $n_0\in E$. Since $n_0$ is arbitrary, the regularity condition holds for all initial states.  By Theorem~\ref{thm:critical_regular}, for every initial state $n_0\in E$, the system converges exponentially fast to a stationary limiting distribution $\pi_{n_0}$. Together with Theorem~\ref{thm:critical_sharpness}, this shows that $\lambda_{\rm gap}$ is the optimal uniform exponential convergence rate.
\end{proof}

To better illustrate Theorem \ref{thm:critical_regular} and Corollary \ref{cor:fully_conservative}, we consider two examples.

\begin{example}\label{ex:telegraph}
Consider the canonical three-stage gene expression model involving gene state switching, transcription, and translation \cite{shahrezaei2008analytical}, with the following reaction scheme:
\begin{equation*}
\begin{gathered}
G_{\mathrm{off}}\xrightarrow{k_{\mathrm{on}}}G_{\mathrm{on}},\quad
G_{\mathrm{on}}\xrightarrow{k_{\mathrm{off}}}G_{\mathrm{off}},\\
G_{\mathrm{on}}\xrightarrow{k_m}G_{\mathrm{on}}+M, \quad M\xrightarrow{k_p}M+P,\\
M\xrightarrow{\gamma_m}\varnothing, \quad P\xrightarrow{\gamma_p}\varnothing.
\end{gathered}
\end{equation*}
Here, the gene of interest switches between the active state $G_{\mathrm{on}}$ and the inactive state $G_{\mathrm{off}}$. The first two reactions describe gene state switching, the middle two reactions describe the synthesis of mRNA $M$ and protein $P$, and the last two reactions describe the degradation of mRNA and protein. Transcription occurs only when the gene is in the active state. Note that in Example~\ref{ex:gene_expression}, mRNA is produced constitutively. In contrast, the current example captures the bursty production of mRNA \cite{shahrezaei2008analytical, jia2017simplification}.

With $S_1=G_{\mathrm{off}}$, $S_2=G_{\mathrm{on}}$, $S_3=M$, and $S_4=P$, the first-order influx matrix $A$ and the zero-order influx vector $b$ are given by
\begin{equation*}
A = \begin{pmatrix}
-k_{\mathrm{on}} & k_{\mathrm{off}} & 0 & 0\\
k_{\mathrm{on}} & -k_{\mathrm{off}} & 0 & 0\\
0 & k_m & -\gamma_m & 0\\
0 & 0 & k_p & -\gamma_p
\end{pmatrix}, \quad
b = \begin{pmatrix}
0 \\ 0 \\ 0 \\ 0
\end{pmatrix}.
\end{equation*}
It is clear that $A$ has a zero eigenvalue, while its nonzero eigenvalues are $-(k_{\mathrm{on}}+k_{\mathrm{off}})$, $-\gamma_m$, and $-\gamma_p$. Therefore, the zero eigenvalue is simple, and
\begin{equation*}
\lambda_{\mathrm{gap}} = \min\{k_{\mathrm{on}}+k_{\mathrm{off}},\gamma_m,\gamma_p\}.
\end{equation*}
The left nullspace of $A$ is one-dimensional and is spanned by $\alpha = (1,1,0,0)$. It is straightforward to verify that conditions (C1) and (C2) are satisfied and that $\alpha$ is orthogonal to all reaction vectors. Therefore, every left null vector of $A$ is a conservation law.

Suppose that there is only one gene copy within a cell. For any initial state $n_0\in E$ satisfying $G_{\mathrm{off}}+G_{\mathrm{on}}=1$, the system evolves within the closed irreducible class
\begin{equation*}
\mathcal{D} = \{n = (n_1,n_2,n_3,n_4)\in E:n_1+n_2=1\},
\end{equation*}
which admits a unique stationary distribution $\pi_{n_0}$. By Corollary~\ref{cor:fully_conservative}, the time-dependent distribution $p_t$ converges exponentially fast to $\pi_{n_0}$ under both the $L^1$-Wasserstein and total variation distances, with optimal exponential convergence rate $\lambda_{\mathrm{gap}}$. Hence the system is exponentially ergodic within this positive recurrent class. We emphasize that the stationary mRNA distribution for the three-stage gene expression model can be obtained in closed form \cite{peccoud1995markovian}, whereas the exact protein distribution remains an open problem.
\end{example}

\begin{example}\label{ex:closed_monomolecular}
In Example \ref{ex:open_monomolecular}, we investigated the long-time behavior of the stochastic dynamics of an open monomolecular reaction network. Here, we consider a closed monomolecular reaction network consisting only of conversion reactions $S_i\to S_j$. Recall that a reaction network is called \emph{weakly reversible} if every connected component of the reaction graph is strongly connected, i.e. every connected component of the underlying undirected reaction graph is also a connected component of the directed reaction graph \cite{anderson2010product}. For simplicity, we assume that the closed monomolecular network under consideration is weakly reversible.

Let $V_1,\cdots,V_m$ denote the connected components of the underlying undirected reaction graph. All reactions within $V_\ell$ have the form
\begin{equation*}
S_i\xrightarrow{k_{ij}}S_j,\quad i,j\in V_\ell,\quad i\ne j,
\end{equation*}
where $k_{ij}\ge0$, with $k_{ij}=0$ if the corresponding reaction is absent. Since no reaction connects two distinct connected components, after relabeling the species, the first-order influx matrix $A$ and the zero-order influx vector $b$ can be written as
\begin{equation*}
A = \begin{pmatrix}
A_1 & 0 & \cdots & 0\\
0 & A_2 & \cdots & 0\\
\vdots & \vdots & \ddots & \vdots\\
0 & 0 & \cdots & A_m
\end{pmatrix},
\quad
b=0,
\end{equation*}
where each block $A_\ell$ is indexed by $V_\ell$ and satisfies
\begin{equation*}
(A_\ell)_{ji} = k_{ij}, \quad i\ne j, \quad
(A_\ell)_{ii} = -\sum_{\substack{j\in V_\ell,\;j\ne i}}k_{ij}.
\end{equation*}
Since the reaction network is weakly reversible, each $V_\ell$ is strongly connected, and hence $A_\ell$ is an irreducible Metzler matrix. Moreover, since the columns of $A_\ell$ sum to zero, $A_\ell^{\mathsf T}$ is the generator matrix of an irreducible finite-state continuous-time Markov chain describing the conversion of species within $V_\ell$. The Perron-Frobenius theorem then implies that zero is a simple eigenvalue of $A_\ell$, while all other eigenvalues have negative real parts. Hence, $\lambda_{\max}=0$, and the algebraic and geometric multiplicities of the zero eigenvalue of $A$ are both equal to $m$. Thus, condition (C1) holds. Since the reaction network contains no zero-order reactions, we have $b=0$, and condition (C2) also holds.

Since the zero eigenvalue of each block $A_\ell$ is simple, we have
\begin{equation*}
\ker(A^{\mathsf T}) = \operatorname{span}\{\mathbf{1}_{V_1},\cdots,\mathbf{1}_{V_m}\},
\end{equation*}
where $\mathbf{1}_{V_{\ell}}$ is the vector whose components are $1$ for indices in $V_{\ell}$ and $0$ otherwise. Every conversion reaction preserves the total number of molecules within its connected component. Therefore, every left null vector $\alpha\in\ker(A^{\mathsf T})$ of $A$ is a conservation law. For any initial state $n_0\in E$, the total number of molecules
\begin{equation*}
N_\ell = \sum_{i\in V_\ell}(n_{0})_i, \quad \ell=1,\cdots,m,
\end{equation*}
in each connected component $V_\ell$ is conserved. Hence, the system evolves within the closed irreducible class
\begin{equation*}
\mathcal{D} = \{n\in E:\sum_{i\in V_\ell}n_i=N_\ell,\;\ell=1,\cdots,m\},
\end{equation*}
which admits a unique stationary distribution $\pi_{n_0}$. By Corollary~\ref{cor:fully_conservative}, the time-dependent distribution $p_t$ converges exponentially fast to the stationary distribution $\pi_{n_0}$ under both the $L^1$-Wasserstein and total variation distances, with optimal exponential convergence rate $\lambda_{\mathrm{gap}}$. In fact, the stationary distribution $\pi_{n_0}$ of the stochastic system has the following multinomial-product form \cite{jahnke2007solving, anderson2010product}:
\begin{equation*}
\pi_{n_0}(n) = \prod_{\ell:N_\ell>0}\left[N_\ell!\prod_{i\in V_\ell}\frac{(z_i^*/N_\ell)^{n_i}}{n_i!}\right],
\quad n\in\mathcal{D}.
\end{equation*}
\end{example}

We have shown that the regularity condition in Definition \ref{regular} is crucial for exponential convergence of the stochastic dynamics in the critical regime. The following theorem characterizes the long-time behavior of the system when the regularity condition is violated.

\begin{theorem}\label{thm:critical_irregular}
Fix an initial state $n_0\in E$. Assume that $\lambda_{\max}=0$, that conditions (C1) and (C2) are satisfied, and that the regularity condition in Definition~\ref{regular} is violated. Then the time-dependent distribution $p_t$ converges in total variation to a stationary limiting distribution $\pi_{n_0}$ supported on the absorbing set $\mathcal{B}$, i.e.
\begin{equation*}
\lim_{t\to\infty} \bigl\|p_t-\pi_{n_0}\bigr\|_{\mathrm{TV}}=0.
\end{equation*}
However, the convergence does not hold under the $L^1$-Wasserstein distance, i.e. $W_1(p_t,\pi_{n_0})\not\to0$ as $t\to\infty$. Moreover, exponential convergence fails under both the $L^1$-Wasserstein and total variation distances.
\end{theorem}

\begin{proof}
The proof is given in Section~\ref{sec:critical_irregular_proof} using martingale and extinction-time arguments.
\end{proof}

Theorems \ref{thm:critical_regular} and \ref{thm:critical_irregular} together provide a complete characterization of the long-time behavior of the stochastic system in the critical regime. Specifically, when conditions (C1) and (C2) hold, exponential convergence occurs if and only if the regularity condition in Definition \ref{regular} is satisfied. When the regularity condition is broken, the system still converges under the total variation distance, but the convergence is no longer exponential and fails to hold under the $L^1$-Wasserstein distance.

\subsection{Long-time behavior in the divergent regimes}
Proposition~\ref{prop:det_longtime} shows that the deterministic dynamics of a linear reaction network fails to converge to a finite equilibrium if and only if one of the following two conditions is satisfied:
\begin{enumerate}
\item[(D1)] The system is in the unstable regime with $\lambda_{\max}>0$ (exponential divergence);
\item[(D2)] The system is in the critical regime with $\lambda_{\max}=0$, and at least one of the conditions (C1) and (C2) is violated (polynomial divergence).
\end{enumerate}
In either case, for any initial condition $z(0)>0$, the deterministic trajectory diverges to infinity, i.e.
\begin{equation*}
\|z(t)\|_1\to\infty,\quad\text{as }t\to\infty.
\end{equation*}

We now investigate the long-time behavior of the stochastic dynamics. In the divergence regime (D1) or (D2), the total mass $\|X_t\|_1$ does not necessarily diverge to infinity: instead, the process may hit the boundary of the state space, where some species become extinct and can never recover. The following two theorems treat the polynomially and exponentially divergent regimes separately. They show that, in both cases, no closed irreducible positive recurrent class can contain an interior state.

\begin{theorem}\label{thm:divergent_D1}
Assume that condition (D1) is satisfied. Then every interior state $n_0>0$ is transient. Moreover,
\begin{equation*}
\mathbb{P}_{n_0}\bigl(\|X_t\|_1\to\infty\text{ or }X_t\text{ is absorbed on }\partial E\bigr) = 1 .
\end{equation*}
In particular, every stationary distribution of the system, if it exists, is supported on the boundary of the state space, i.e.
\begin{equation*}
\partial E = \{n\in E:\;n_i= 0\text{ for at least one }i\}.
\end{equation*}
Equivalently, every closed irreducible positive recurrent class of the system, if it exists, is contained in $\partial E$.
\end{theorem}

\begin{proof}
The proof of the theorem is given in Supplementary Section~S4.
\end{proof}

\begin{theorem}\label{thm:divergent_D2}
Assume that condition (D2) is satisfied. Then every interior state $n_0>0$ is either null recurrent or transient. Moreover, every stationary distribution of the system, if it exists, is supported on the boundary $\partial E$ of the state space. Equivalently, every closed irreducible positive recurrent class of the system, if it exists, is contained in $\partial E$.
\end{theorem}

\begin{proof}
The proof of the theorem is given in Supplementary Section~S4.
\end{proof}

According to the above two theorems, the divergent regimes (D1) and (D2) may lead to different long-time behaviors. In the divergent regime (D1), every interior state is transient: starting from an interior state, the system either diverges to infinity or, if the boundary is accessible, is eventually absorbed on the boundary. However, in the divergent regime (D2), every interior state is either transient or null recurrent. If an interior state is null recurrent, then starting from this state, the system remains within a null recurrent class, without converging to a stationary distribution and without escaping to infinity.

In particular, we emphasize that, in both divergent regimes (D1) and (D2), although the deterministic trajectories grow unboundedly, the stochastic dynamics may instead be absorbed on the boundary of the state space when the initial state is transient. Therefore, deterministic and stochastic dynamics may exhibit fundamentally different long-time behaviors. We now illustrate this phenomenon with two examples.

\begin{example}\label{ex:lambda_pos}
Consider the reaction network
\begin{equation*}
\varnothing \xrightarrow{k_0} S_1,\quad
S_1 \xrightarrow{k_1} S_2,\quad
S_2 \xrightarrow{k_2} S_1+2S_2,\quad
S_2 \xrightarrow{k_3} \varnothing,
\end{equation*}
where $k_0\ge0$, $k_1>0$, and $k_2>k_3>0$. For this network, the first-order influx matrix $A$ and the zero-order influx vector $b$ are given by
\begin{equation*}
A = \begin{pmatrix}
-k_1 & k_2\\
k_1 & k_2-k_3
\end{pmatrix}, \quad
b = \begin{pmatrix}
k_0 \\ 0
\end{pmatrix}.
\end{equation*}
It is straightforward to verify that $\det(A)=k_1(k_3-2k_2)<0$. Therefore, $A$ has one positive and one negative eigenvalue, and hence $\lambda_{\max}>0$. Thus, the network satisfies the divergence condition (D1). According to Proposition~\ref{prop:det_longtime}, every deterministic trajectory with a positive initial condition diverges to infinity.

We first consider the case of $k_0>0$, i.e. the first reaction is present. In this case, although the boundary $\partial E$ of the state space is accessible, it contains no closed communicating class. Indeed, from every boundary state, there exists a finite sequence of reactions that leads to an interior state: the synthesis reaction $\varnothing\rightarrow S_1$ produces $S_1$, the conversion reaction $S_1\rightarrow S_2$ produces $S_2$, and the autocatalytic reaction $S_2\to S_1+2S_2$ can subsequently produce both species simultaneously. Therefore, starting from any interior initial state $n_0>0$, the system cannot be absorbed on the boundary. According to Theorem~\ref{thm:divergent_D1}, the stochastic trajectories must also diverge to infinity, i.e.
\begin{equation*}
\lim_{t\to\infty}\|X_t\|_1 = \infty,
\end{equation*}
almost surely. In this case, the deterministic and stochastic systems exhibit the same long-time behavior.

We next consider the case of $k_0=0$, i.e. the first reaction is absent. In this case, $(0,0)$ is an absorbing state of the stochastic system and can be reached from every nonzero state through a finite sequence of reactions with strictly positive transition rates. Moreover, it is straightforward to verify that $(0,0)$ is the only closed class contained in the boundary $\partial E$ of the state space. Hence, Theorem~\ref{thm:divergent_D1} implies that the process either reaches the absorbing state $(0,0)$ in finite time or diverges to infinity, i.e. $\|X_t\|_1\to\infty$ as $t\to\infty$. In this case, the deterministic and stochastic systems may exhibit different long-time behaviors.
\end{example}

\begin{example}\label{ex:jordan}
Consider the reaction network
\begin{equation*}
\varnothing \xrightarrow{k_0} S_1,\quad
S_1 \xrightarrow{k_1} 2S_1,\quad
S_1 \xrightarrow{k_1} \varnothing,
\end{equation*}
with $k_0,k_1>0$. Here, the rate constants of the last two first-order reactions are assumed to be identical. For this network, the first-order influx matrix is $A=[0]$ (a $1\times1$ zero matrix), and the zero-order influx vector is $b=k_0$. It is clear that the zero eigenvalue of $A$ is simple. However, $b$ is not orthogonal to the left nullspace of $A$, which coincides with the entire real axis $\mathbb{R}$. Therefore, the network satisfies the divergence condition (D2).

The stochastic process is a birth-death chain on the set of nonnegative integers, with birth rates $\lambda_n=k_0+k_1n$ and death rates $\mu_n=k_1n$. Classical results for birth-death processes \cite{karlin1975first} state that such a chain is
\begin{itemize}
\item[(i)] recurrent if and only if
\begin{equation}\label{eq:recurrence_crit}
\sum_{n=1}^{\infty}\frac{\mu_1\mu_2\cdots\mu_n}{\lambda_1\lambda_2\cdots\lambda_n} = \infty;
\end{equation}
\item[(ii)] positive recurrent if and only if
\begin{equation}\label{eq:positive_recurrence_crit}
\sum_{n=0}^{\infty}\frac{\lambda_0\lambda_1\cdots\lambda_{n-1}}{\mu_1\mu_2\cdots\mu_n} < \infty.
\end{equation}
\end{itemize}
Applying these criteria to the current system, we obtain
\begin{equation*}
\begin{gathered}
\frac{\mu_1\cdots\mu_n}{\lambda_1\cdots\lambda_n}
= \prod_{i=1}^{n}\frac{k_1i}{k_0+k_1i} \sim n^{-k_0/k_1},\\
\frac{\lambda_0\cdots\lambda_{n-1}}{\mu_1\cdots\mu_n}
= \frac{k_0}{k_1n}\prod_{i=1}^{n-1}\frac{k_0+k_1i}{k_1i}
\sim n^{k_0/k_1-1}.
\end{gathered}
\end{equation*}
Consequently,
\begin{itemize}
\item[(i)] when $k_0\le k_1$, both series \eqref{eq:recurrence_crit} and \eqref{eq:positive_recurrence_crit} diverge. In this case, the system is null recurrent;
\item[(ii)] when $k_0>k_1$, the series \eqref{eq:recurrence_crit} converges, while the series \eqref{eq:positive_recurrence_crit} diverges. In this case, the system is transient, and hence $\|X_t\|_1\to\infty$ as $t\to\infty$.
\end{itemize}
In neither case does a stationary distribution exist. This example illustrates Theorem~\ref{thm:divergent_D2}: an interior state $n_0>0$ can be either null recurrent or transient.
\end{example}

\subsection{Summary of main results}
In summary, we have provided a spectral and structural characterization of the long-time behavior of reducible stochastic linear reaction networks based on (i) the spectrum of the first-order influx matrix $A$, (ii) the position of the zero-order influx vector $b$ relative to the left nullspace of $A$, and (iii) whether every left null vector of $A$ corresponds to a conservation law when the reaction network is restricted to the set of persistent species. In the stable regime, the deterministic and stochastic systems exhibit similar long-time behaviors: the former converges to a globally attracting equilibrium exponentially fast, while the latter converges to a unique stationary distribution supported on the absorbing set $\mathcal{B}$ exponentially fast. In contrast, in the critical and divergent regimes, the deterministic and stochastic systems may exhibit fundamentally different long-time behaviors. The main results are summarized in Table~\ref{tab:comparison}.

\begin{table}[htb!]
\footnotesize
\renewcommand{\arraystretch}{1.5}
\begin{tabular}{@{}p{0.23\textwidth}p{0.29\textwidth}p{0.31\textwidth}p{0.1\textwidth}@{}}
\toprule
Regime & Deterministic system & Stochastic system & Reference \\
\midrule
All cases
&
The solution exists globally, i.e. $z(t)$ remains finite for all $t\ge0$ and all initial conditions.
&
The system is non-explosive, i.e. $J_\infty=\infty$ almost surely for all initial states
&
Theorem~\ref{thm:non-explosive}
\\
\midrule
Stable case of $\lambda_{\max}<0$
&
Exponential convergence to a globally attracting equilibrium $z^*\geq0$, at rate
$-\lambda_{\max}$
&
Exponential convergence to a unique stationary distribution $\pi$ supported on the absorbing set $\mathcal{B}$, at rate
$-\lambda_{\max}$
&
Theorem~\ref{thm:stable_case}
\\
\addlinespace

Critical case of $\lambda_{\max}=0$, conditions (C1) and (C2) hold; regularity condition is satisfied
&
Exponential convergence to an equilibrium $z^*\geq0$ that depends on the initial condition, at rate $\lambda_{\mathrm{gap}}$
&
Exponential convergence to a stationary distribution $\pi$ that depends on the initial state, at rate $\lambda_{\mathrm{gap}}$
&
Theorem~\ref{thm:critical_regular}
\\
\addlinespace

Critical case of $\lambda_{\max}=0$, conditions (C1) and (C2) hold; regularity condition is violated
&
Exponential convergence to an equilibrium $z^*\geq0$ that depends on the initial condition, at rate $\lambda_{\mathrm{gap}}$
&
Non-exponential convergence to a stationary distribution $\pi$ that depends on the initial state under the total variation distance; convergence fails under the $L^1$-Wasserstein distance
&
Theorem~\ref{thm:critical_irregular}
\\
\addlinespace

Critical case of $\lambda_{\max}=0$, condition (C1) or (C2) does not hold
&
Polynomial divergence to infinity
&
Every interior state is transient or null recurrent; any stationary distribution, if it exists, is supported on the boundary
&
Theorem~\ref{thm:divergent_D2}
\\
\addlinespace

Unstable case of $\lambda_{\max}>0$
&
Exponential divergence to infinity
&
Every interior state is transient; any stationary distribution, if it exists, is supported on the boundary
&
Theorem~\ref{thm:divergent_D1}
\\
\bottomrule
\end{tabular}
\caption{\textbf{Comparison of the long-time behaviors of deterministic and stochastic dynamics for linear reaction networks.}}\label{tab:comparison}
\end{table}

\section{Proof of Theorem~\ref{thm:stable_case}}\label{sec:stable_proof}
Throughout this section, we assume $\lambda_{\max}<0$ and write
\begin{equation*}
\rho:=-\lambda_{\max}>0.
\end{equation*}
By Theorem~\ref{thm:hitting_exp}, the absorbing set $\mathcal B$ defined in \eqref{eq:def_R} is a closed irreducible positive recurrent class. Hence the process admits a unique stationary distribution $\pi=(\pi(n))_{n\in E}$ supported on $\mathcal B$. For any $\varepsilon\in(0,\rho)$, the matrix $A+(\rho-\varepsilon)I$ is a stable Metzler matrix. A standard characterization of stable Metzler matrices \cite{berman1994nonnegative} guarantees that there exists a strictly positive column vector $v>0$ such that
\begin{equation}\label{eq:weight_ineq}
v^{\mathsf T}A
\le
-(\rho-\varepsilon)v^{\mathsf T}.
\end{equation}
Set $v_{\min}:=\min_i v_i$ and $v_{\max}:=\max_i v_i$.
For a function $h:E\to\mathbb R$, we use the same symbol $Q$ for the generator operator,
\begin{equation*}
(Qh)(n)
:=
\sum_{R\in\mathcal R}
f_R(n)\bigl(h(n+w_R)-h(n)\bigr),
\qquad n\in E,
\end{equation*}
where $w_R$ denotes the reaction vector of $R$.

We first establish the moment bounds needed below.

\begin{proposition}\label{prop:lyapunov_moments}
There exists a constant $\theta>0$ such that
\begin{equation}\label{temp1}
\mathbb E_{\pi}
\bigl[e^{\theta\|X_t\|_1}\bigr]
=
\sum_{n\in E}e^{\theta\|n\|_1}\pi(n)
<
\infty,
\end{equation}
and, for every initial state $n_0\in E$,
\begin{equation}\label{temp2}
\sup_{t\ge0}
\mathbb E_{n_0}
\bigl[e^{\theta\|X_t\|_1}\bigr]
<
\infty.
\end{equation}
In particular,
\begin{equation}\label{temp3}
\mathbb E_{\pi}\bigl[\|X_t\|_1^p\bigr]<\infty,
\qquad
\sup_{t\ge0}\mathbb E_{n_0}\bigl[\|X_t\|_1^p\bigr]<\infty
\end{equation}
for every $n_0\in E$ and $p\ge1$.
\end{proposition}

\begin{proof}
Fix $\varepsilon$ and $v$ as above and define $V(n):=v^{\mathsf T}n$. By \eqref{eq:weight_ineq},
\begin{equation*}
(QV)(n)
=
v^{\mathsf T}(An+b)
\le
-(\rho-\varepsilon)V(n)+\xi,
\qquad
\xi:=v^{\mathsf T}b.
\end{equation*}
Let $q(n):=\sum_{R\in\mathcal R}f_R(n)$ denote the total jump rate. Since the network is linear and $v>0$, there exists $C_0<\infty$ such that
\begin{equation*}
q(n)\le C_0\bigl(1+V(n)\bigr).
\end{equation*}
Moreover, because the reaction set is finite,
\[
\Delta_v
:=
\max_{R\in\mathcal R}|v^{\mathsf T}w_R|
<
\infty.
\]
Choose $\theta_0>0$ sufficiently small that
\begin{equation*}
\theta_0\Delta_v\le1,
\qquad
\gamma
:=
(\rho-\varepsilon)-2\theta_0\Delta_v^2C_0
>
0.
\end{equation*}
Set $W(n):=e^{\theta_0V(n)}$. Using $e^x-1\le x+2x^2$ for $|x|\le1$, we obtain
\begin{equation*}
\begin{aligned}
(QW)(n)
&=
W(n)
\sum_{R\in\mathcal R}
f_R(n)
\left(
e^{\theta_0v^{\mathsf T}w_R}-1
\right)
\\
&\le
\theta_0W(n)
\left[
QV(n)+2\theta_0\Delta_v^2q(n)
\right]
\\
&\le
\theta_0W(n)
\left[
-\gamma V(n)
+\xi
+2\theta_0\Delta_v^2C_0
\right].
\end{aligned}
\end{equation*}
The expression in brackets tends to $-\infty$ as $V(n)\to\infty$. Hence there exist constants $\beta>0$ and $C<\infty$ such that
\begin{equation}\label{eq:stable_exp_drift}
(QW)(n)
\le
-\beta W(n)+C,
\qquad n\in E.
\end{equation}
For $m>\|n_0\|_1$, define
\begin{equation*}
\sigma_m:=\inf\{t\ge0:\|X_t\|_1\ge m\}.
\end{equation*}
Before $\sigma_m$, the process evolves in a finite subset of $E$. Applying the stopped Dynkin formula to the process $e^{\beta t}W(X_t)$ and using \eqref{eq:stable_exp_drift}, we obtain
\begin{equation*}
\begin{aligned}
\mathbb E_{n_0}
\left[
e^{\beta(t\wedge\sigma_m)}
W(X_{t\wedge\sigma_m})
\right]
&\le
W(n_0)
+
C\,
\mathbb E_{n_0}
\int_0^{t\wedge\sigma_m}
e^{\beta s}\,\mathrm ds
\\
&\le
W(n_0)
+
\frac{C}{\beta}
\left(
e^{\beta t}-1
\right).
\end{aligned}
\end{equation*}
By non-explosivity, $\sigma_m\uparrow\infty$ almost surely. Letting $m\to\infty$ and applying Fatou's lemma give
\begin{equation*}
\mathbb E_{n_0}W(X_t)
\le
e^{-\beta t}W(n_0)
+
\frac{C}{\beta}
\left(
1-e^{-\beta t}
\right)
\le
W(n_0)+\frac{C}{\beta}.
\end{equation*}
Therefore $\sup_{t\ge0}\mathbb E_{n_0}W(X_t)<\infty$.

Since $\mathcal B$ is an irreducible positive recurrent class, the drift condition \eqref{eq:stable_exp_drift} and standard Foster-Lyapunov theory \cite{meyn1993stability} further imply $\mathbb E_\pi W<\infty$. Finally, setting $\theta:=\theta_0v_{\min}$,
we have
\[
V(n)\ge v_{\min}\|n\|_1,
\qquad
e^{\theta\|n\|_1}\le W(n).
\]
Hence \eqref{temp1} and \eqref{temp2} follow immediately. This further implies \eqref{temp3}.
\end{proof}

We next construct a coupling for two copies of the process evolving in $\mathcal B$. For brevity, write\
\begin{equation*}
w_{0j}:=\nu_{0j}, \qquad w_{ij}:=\nu_{ij}-e_i
\end{equation*}
for the reaction vectors of $R_{0j}$ and $R_{ij}$, respectively. Let $\gamma$ be any probability distribution on $\mathcal B\times\mathcal B$, and let $(X_0,Y_0)$ have distribution $\gamma$. At a state $(x,y)\in\mathcal B\times\mathcal B$, every zero-order reaction $R_{0j}$ induces the simultaneous transition
\begin{equation*}
(x,y)
\longrightarrow
(x+w_{0j},y+w_{0j})
\qquad
\text{at rate }k_{0j}.
\end{equation*}
For every first-order reaction $R_{ij}$, we use the common-reaction coupling
\begin{equation*}
\begin{array}{lll}
(x,y)\longrightarrow(x+w_{ij},y+w_{ij})
&\text{at rate}&
k_{ij}\min\{x_i,y_i\},
\\[1mm]
(x,y)\longrightarrow(x+w_{ij},y)
&\text{at rate}&
k_{ij}(x_i-y_i)_+,
\\[1mm]
(x,y)\longrightarrow(x,y+w_{ij})
&\text{at rate}&
k_{ij}(y_i-x_i)_+,
\end{array}
\end{equation*}
where $a_+:=\max\{a,0\}$. The total rate of reaction $R_{ij}$ in the first marginal is
\[
k_{ij}\min\{x_i,y_i\}
+
k_{ij}(x_i-y_i)_+
=
k_{ij}x_i,
\]
and similarly the total rate in the second marginal is $k_{ij}y_i$. Thus both marginals have the transition rates of the original process. Since $\mathcal B$ is closed, all transitions remain in $\mathcal B\times\mathcal B$. Moreover, every jump of the coupled process is a jump of at least one marginal, so non-explosivity of the marginals implies non-explosivity of the coupled process.

Define $\Delta_t:=X_t-Y_t$. Simultaneous reactions leave $\Delta_t$ unchanged; only unmatched first-order reactions can change the difference.

\begin{lemma}\label{lem:coupling_drift}
Under the coupling constructed above,
\begin{equation}\label{eq:Delta_t}
\mathbb E\|\Delta_t\|_1
\le
\frac{v_{\max}}{v_{\min}}
\mathbb E\|\Delta_0\|_1
e^{-(\rho-\varepsilon)t},
\qquad t\ge0.
\end{equation}
\end{lemma}

\begin{proof}
If $\mathbb E\|\Delta_0\|_1=\infty$, there is nothing to prove. Hence assume that $\mathbb E\|\Delta_0\|_1<\infty$. For $d\in\mathbb R^N$, define
\[
\Psi(d)
:=
\sum_{i=1}^Nv_i|d_i|,
\qquad
\Psi_t:=\Psi(\Delta_t),
\]
and set $\Phi(x,y):=\Psi(x-y)$. For any $a\in\mathbb R$ and $b\ge0$,
\begin{equation}\label{eq:abs_bound}
|a+b|-|a|\le b,
\qquad
|a-b|-|a|\le b.
\end{equation}
Fix $(x,y)\in\mathcal B\times\mathcal B$ and write $d:=x-y$. Consider an unmatched transition associated with $R_{ij}$. If $d_i>0$, the first coordinate jumps alone, so $d$ changes to $d+w_{ij}$. Since $d_i\ge1$, we have
\begin{equation*}
|d_i+\nu_i^{ij}-1|-|d_i|=\nu_i^{ij}-1.
\end{equation*}
Using \eqref{eq:abs_bound} for all remaining coordinates gives
\begin{equation*}
\begin{aligned}
\Psi(d+w_{ij})-\Psi(d)
&\le
v_i(\nu_i^{ij}-1)
+
\sum_{\ell\ne i}
v_\ell\nu_\ell^{ij}
=
v^{\mathsf T}w_{ij}.
\end{aligned}
\end{equation*}
Similarly, if $d_i<0$, only the second coordinate jumps and $d$ changes to $d-w_{ij}$. Since $d_i\le-1$, we have
\begin{equation*}
|d_i-\nu_i^{ij}+1|-|d_i|=\nu_i^{ij}-1,
\end{equation*}
and therefore
\begin{equation*}
\Psi(d-w_{ij})-\Psi(d)
\le
v^{\mathsf T}w_{ij}.
\end{equation*}

Let $\widetilde Q$ denote the generator of the coupled process. Since simultaneous transitions do not change $\Phi$, the preceding estimates imply
\begin{equation*}
\begin{aligned}
(\widetilde Q\Phi)(x,y)
={}&
\sum_{i=1}^N\sum_{j=1}^{r_i}
\mathbf1_{\{d_i>0\}}
k_{ij}d_i
\bigl[
\Psi(d+w_{ij})-\Psi(d)
\bigr]
\\
&+
\sum_{i=1}^N\sum_{j=1}^{r_i}
\mathbf1_{\{d_i<0\}}
k_{ij}|d_i|
\bigl[
\Psi(d-w_{ij})-\Psi(d)
\bigr]
\\
\le{}&
\sum_{i=1}^N
|d_i|\,
v^{\mathsf T}
\sum_{j=1}^{r_i}
k_{ij}w_{ij}
=
\sum_{i=1}^N
|d_i|
(v^{\mathsf T}A)_i
\\
\le{}&
-(\rho-\varepsilon)
\Psi(d),
\end{aligned}
\end{equation*}
where the last inequality follows from \eqref{eq:weight_ineq}. To proceed, we define
\[
\sigma_m
:=
\inf\left\{
t\ge0:
\|X_t\|_1+\|Y_t\|_1\ge m
\right\}.
\]
Applying the stopped Dynkin formula to the process $e^{(\rho-\varepsilon)t}\Phi(X_t,Y_t)$ gives
\begin{equation*}
\mathbb E
\left[
e^{(\rho-\varepsilon)(t\wedge\sigma_m)}
\Psi_{t\wedge\sigma_m}
\right]
\le
\mathbb E\Psi_0.
\end{equation*}
The coupled process is non-explosive, so $\sigma_m\uparrow\infty$ almost surely. Letting $m\to\infty$ and applying Fatou's lemma yield
\[
\mathbb E\Psi_t
\le
e^{-(\rho-\varepsilon)t}
\mathbb E\Psi_0.
\]
Finally, $v_{\min}\|d\|_1\le\Psi(d)\le v_{\max}\|d\|_1$, and \eqref{eq:Delta_t} follows.
\end{proof}

For $t\ge0$, let $P_t(n,\cdot)$ denote the law of $X_t$ under the initial condition $X_0=n$. More generally, for an initial distribution $\mu$, define
\begin{equation*}
P_t(\mu,\cdot)
:=
\sum_{n\in E}
P_t(n,\cdot)\mu(n).
\end{equation*}
Hence $P_t(\delta_n,\cdot)=P_t(n,\cdot)$, and the distribution $p_t$ used in Section~\ref{sec:prelim} is precisely $P_t(n_0,\cdot)$. The preceding coupling immediately gives exponential convergence within $\mathcal B$.

\begin{proposition}\label{prop:W1_upper}
For every $\varepsilon\in(0,\rho)$ and $n_0\in\mathcal B$, there exists $C(n_0,\varepsilon)<\infty$ such that
\begin{equation*}
W_1
\bigl(
P_t(n_0,\cdot),\pi
\bigr)
\le
C(n_0,\varepsilon)
e^{-(\rho-\varepsilon)t},
\qquad t\ge0.
\end{equation*}
Moreover, there exist constants $C_1(\varepsilon),C_2(\varepsilon)<\infty$, independent of $n_0$, such that
\begin{equation}\label{eq:C_def_stable}
C(n_0,\varepsilon)
\le
C_1(\varepsilon)\|n_0\|_1
+
C_2(\varepsilon).
\end{equation}
\end{proposition}

\begin{proof}
Let $X_0=n_0$ and let $Y_0$ have distribution $\pi$. Apply the preceding coupling to $(X_0,Y_0)$. Since $\pi$ is a stationary distribution, $Y_t\sim\pi$ for every $t\ge0$. Lemma~\ref{lem:coupling_drift} and the coupling characterization of $W_1$ therefore give
\begin{equation*}
\begin{aligned}
W_1
\bigl(
P_t(n_0,\cdot),\pi
\bigr)
&\le
\mathbb E\|X_t-Y_t\|_1
\\
&\le
\frac{v_{\max}}{v_{\min}}
\mathbb E_\pi
\|n_0-Y_0\|_1
e^{-(\rho-\varepsilon)t}
\\
&\le
\frac{v_{\max}}{v_{\min}}
\left(
\|n_0\|_1+
\mathbb E_\pi\|Y_0\|_1
\right)
e^{-(\rho-\varepsilon)t}.
\end{aligned}
\end{equation*}
Proposition~\ref{prop:lyapunov_moments} gives $\mathbb E_\pi\|Y_0\|_1<\infty$, so \eqref{eq:C_def_stable} follows.
\end{proof}

The preceding proposition controls the dynamics after the process has entered $\mathcal B$. To extend the estimate to arbitrary initial states, we need the following lemma.

\begin{lemma}\label{lem:weighted_hitting_stable}
For every $\varepsilon\in(0,\rho)$, there exist a function $U:E\to[1,\infty)$ and a constant $c_U>0$ such that
\begin{equation}\label{eq:U_inequality}
U(n)
\ge
c_U(1+\|n\|_1),
\qquad n\in E,
\end{equation}
and, for every $n\in E$ and $t\ge0$,
\begin{equation}\label{eq:stable_weighted_hitting}
\mathbb E_n
\left[
e^{(\rho-\varepsilon)\tau_{\mathcal B}}
U(X_{\tau_{\mathcal B}})
\right]
\le
U(n),
\qquad
\mathbb E_n
\left[
\mathbf1_{\{\tau_{\mathcal B}>t\}}
U(X_t)
\right]
\le
e^{-(\rho-\varepsilon)t}U(n).
\end{equation}
\end{lemma}

\begin{proof}
Let
\[
I_0:=\{i:z_i^*=0\},
\qquad
\mathcal T:=E\setminus\mathcal B.
\]
If $I_0=\varnothing$, then $\mathcal B=E$ and $\tau_{\mathcal B}=0$, so the conclusion holds with $U(n)=1+\|n\|_1$. Hence assume $I_0\ne\varnothing$. For $i\in I_0$, the $i$th equilibrium equation gives
\begin{equation*}
0=(Az^*+b)_i=b_i+\sum_{\ell\in I_+}a_{i\ell}z_\ell^*.
\end{equation*}
Since every term on the right-hand side is non-negative and $z_\ell^*>0$ for $\ell\in I_+$, we have $b_i=0$ and $a_{i\ell}=0$ for $i\in I_0$ and $\ell\in I_+$. After ordering the species so that $I_0$ precedes $I_+$,
\begin{equation*}
A=
\begin{pmatrix}
A_{I_0}&0\\
A_{I_+,I_0}&A_{I_+}
\end{pmatrix}.
\end{equation*}
Thus $\sigma(A_{I_0})\subseteq\sigma(A)$, so the spectral bound of $A_{I_0}$ is at most $-\rho$. By the characterization of stable Metzler matrices, there exist strictly positive vectors $\widehat v\in\mathbb R^N$ and $v_0\in\mathbb R^{|I_0|}$ such that
\begin{equation*}
\widehat v^{\mathsf T}A
\le
-\left(\rho-\frac{\varepsilon}{2}\right)\widehat v^{\mathsf T},
\qquad
v_0^{\mathsf T}A_{I_0}
\le
-\left(\rho-\frac{\varepsilon}{2}\right)v_0^{\mathsf T}.
\end{equation*}
Define $V(n):=\widehat v^{\mathsf T}n$ and $V_0(n):=v_0^{\mathsf T}n_{I_0}$. Then $V_0(n)>0$ on $\mathcal T$ and $V_0(n)=0$ on $\mathcal B$, while
\begin{equation*}
QV(n)
\le
-\left(\rho-\frac{\varepsilon}{2}\right)V(n)+\xi,
\qquad
QV_0(n)
\le
-\left(\rho-\frac{\varepsilon}{2}\right)V_0(n),
\end{equation*}
where $\xi:=\widehat v^{\mathsf T}b$.

Choose $M>0$ such that $\frac{\varepsilon}{2}M\ge\xi+\rho-\varepsilon$ and set $\mathcal L_M:=\{n\in E:V(n)\le M\}$. Since $\widehat v>0$, the set $\mathcal L_M$ is finite. If $\mathcal L_M\cap\mathcal T\ne\varnothing$, define
\begin{equation*}
V_{0,\min}
:=
\min_{n\in\mathcal L_M\cap\mathcal T}V_0(n)>0
\end{equation*}
and choose $\zeta>0$ such that $\zeta\frac{\varepsilon}{2}V_{0,\min}\ge\xi+\rho-\varepsilon$; if $\mathcal L_M\cap\mathcal T=\varnothing$, choose any $\zeta>0$. Set
\begin{equation*}
U(n):=V(n)+\zeta V_0(n)+1.
\end{equation*}
For $n\in\mathcal T$,
\begin{equation*}
QU(n)+(\rho-\varepsilon)U(n)
\le
-\frac{\varepsilon}{2}V(n)
-\zeta\frac{\varepsilon}{2}V_0(n)
+\xi+\rho-\varepsilon.
\end{equation*}
If $V(n)>M$, then the first negative term makes the right-hand side non-positive; if $V(n)\le M$, then $n\in\mathcal L_M\cap\mathcal T$ and $V_0(n)\ge V_{0,\min}$, so the second negative term does so. Hence
\begin{equation}\label{eq:U_drift}
QU(n)+(\rho-\varepsilon)U(n)\le0,
\qquad n\in\mathcal T.
\end{equation}
Moreover, if $\widehat v_{\min}:=\min_i\widehat v_i$ and $c_U:=\min\{1,\widehat v_{\min}\}$, then
\begin{equation*}
U(n)
\ge
1+\widehat v_{\min}\|n\|_1
\ge
c_U(1+\|n\|_1),
\end{equation*}
which proves \eqref{eq:U_inequality}.

It remains to derive the hitting-time estimates. Let $\sigma_m:=\inf\{s\ge0:\|X_s\|_1\ge m\}$. Applying the stopped Dynkin formula to the process $e^{(\rho-\varepsilon)s}U(X_s)$ at $\tau_{\mathcal B}\wedge t\wedge\sigma_m$ and using \eqref{eq:U_drift} give
\begin{equation*}
\mathbb E_n
\left[
e^{(\rho-\varepsilon)(\tau_{\mathcal B}\wedge t\wedge\sigma_m)}
U\left(
X_{\tau_{\mathcal B}\wedge t\wedge\sigma_m}
\right)
\right]
\le
U(n).
\end{equation*}
By non-explosivity, $\sigma_m\uparrow\infty$ almost surely, and Fatou's lemma yields
\begin{equation*}
\mathbb E_n
\left[
e^{(\rho-\varepsilon)(\tau_{\mathcal B}\wedge t)}
U\left(
X_{\tau_{\mathcal B}\wedge t}
\right)
\right]
\le
U(n).
\end{equation*}
Since $U\ge1$, letting $t\to\infty$ shows that $\tau_{\mathcal B}<\infty$ almost surely and gives
\begin{equation*}
\mathbb E_n
\left[
e^{(\rho-\varepsilon)\tau_{\mathcal B}}
U(X_{\tau_{\mathcal B}})
\right]
\le
U(n).
\end{equation*}
On the other hand, restricting the preceding stopped estimate to $\{\tau_{\mathcal B}>t\}$ gives
\begin{equation*}
e^{(\rho-\varepsilon)t}
\mathbb E_n
\left[
\mathbf1_{\{\tau_{\mathcal B}>t\}}U(X_t)
\right]
\le
U(n).
\end{equation*}
These are precisely the two estimates in \eqref{eq:stable_weighted_hitting}.
\end{proof}

We are now in a position to prove Theorem~\ref{thm:stable_case}.

\begin{proof}[Proof of Theorem~\ref{thm:stable_case}]
Fix $\varepsilon\in(0,\rho)$. By Proposition~\ref{prop:W1_upper}, there exist $C_1(\varepsilon),C_2(\varepsilon)<\infty$ such that
\begin{equation}\label{eq:W1_on_R}
W_1
\bigl(
P_t(n,\cdot),\pi
\bigr)
\le
\left[
C_1(\varepsilon)\|n\|_1+
C_2(\varepsilon)
\right]
e^{-(\rho-\varepsilon)t}
\end{equation}
for every $n\in\mathcal B$ and $t\ge0$. By the strong Markov property at $\tau_{\mathcal B}$,
\begin{equation*}
\begin{aligned}
P_t(n_0,\cdot)
={}&
\mathbb E_{n_0}
\left[
\mathbf1_{\{\tau_{\mathcal B}\le t\}}
P_{t-\tau_{\mathcal B}}
(X_{\tau_{\mathcal B}},\cdot)
\right]
+
\mathbb E_{n_0}
\left[
\mathbf1_{\{\tau_{\mathcal B}>t\}}
\delta_{X_t}
\right].
\end{aligned}
\end{equation*}
Convexity of $W_1$ under mixtures therefore gives
\begin{equation}\label{eq:split}
\begin{aligned}
W_1
\bigl(
P_t(n_0,\cdot),\pi
\bigr)
\le{}&
\mathbb E_{n_0}
\left[
\mathbf1_{\{\tau_{\mathcal B}\le t\}}
W_1
\left(
P_{t-\tau_{\mathcal B}}
(X_{\tau_{\mathcal B}},\cdot),
\pi
\right)
\right]
\\
&+
\mathbb E_{n_0}
\left[
\mathbf1_{\{\tau_{\mathcal B}>t\}}
W_1(\delta_{X_t},\pi)
\right].
\end{aligned}
\end{equation}

We first prove part~(i). Let $U$ and $c_U$ be given by Lemma~\ref{lem:weighted_hitting_stable}. It follows from \eqref{eq:W1_on_R}, \eqref{eq:U_inequality}, and \eqref{eq:stable_weighted_hitting} that
\begin{equation*}
\begin{aligned}
&\mathbb E_{n_0}
\left[
\mathbf1_{\{\tau_{\mathcal B}\le t\}}
W_1
\left(
P_{t-\tau_{\mathcal B}}
(X_{\tau_{\mathcal B}},\cdot),
\pi
\right)
\right]
\\
&\qquad\le
\frac{
C_1(\varepsilon)+C_2(\varepsilon)
}{c_U}
e^{-(\rho-\varepsilon)t}
\mathbb E_{n_0}
\left[
e^{(\rho-\varepsilon)\tau_{\mathcal B}}
U(X_{\tau_{\mathcal B}})
\right]
\\
&\qquad\le
\frac{
C_1(\varepsilon)+C_2(\varepsilon)
}{c_U}
U(n_0)
e^{-(\rho-\varepsilon)t}.
\end{aligned}
\end{equation*}
By Proposition~\ref{prop:lyapunov_moments}, $m_\pi:=\sum_{y\in E}\|y\|_1\pi(y)<\infty$.
For any $x\in E$,
\[
W_1(\delta_x,\pi)
=
\sum_{y\in E}
\|x-y\|_1\pi(y)
\le
\|x\|_1+m_\pi.
\]
Hence
\begin{equation*}
\begin{aligned}
\mathbb E_{n_0}
\left[
\mathbf1_{\{\tau_{\mathcal B}>t\}}
W_1(\delta_{X_t},\pi)
\right]
\le
\frac{1+m_\pi}{c_U}
\mathbb E_{n_0}
\left[
\mathbf1_{\{\tau_{\mathcal B}>t\}}
U(X_t)
\right]
\le
\frac{1+m_\pi}{c_U}
U(n_0)
e^{-(\rho-\varepsilon)t}.
\end{aligned}
\end{equation*}
Substituting the preceding two estimates into \eqref{eq:split} yields
\[
W_1
\bigl(
P_t(n_0,\cdot),\pi
\bigr)
\le
C(n_0,\varepsilon)
e^{-(\rho-\varepsilon)t},
\qquad t\ge0,
\]
which proves part~(i).

We next prove part~(ii). By the Perron-Frobenius theorem for stable Metzler matrices, there exists $u\ge0$, $u\ne0$, such that $A^{\mathsf T}u = -\rho u$. Proposition~\ref{prop:lyapunov_moments} implies $m_u:=\mathbb E_\pi[u^{\mathsf T}X]<\infty$.
Choose $i$ such that $u_i>0$ and then choose an integer $M$ sufficiently large that $\Delta:=u^{\mathsf T}(Me_i)-m_u>0$.
Set $n_0^*:=Me_i$. Since the mean dynamics satisfy
\[
\frac{\mathrm d}{\mathrm dt}
\mathbb E X_t
=
A\mathbb E X_t+b,
\]
and the stationary mean satisfies the same equation at equilibrium,
\begin{equation*}
\begin{aligned}
\mathbb E_{n_0^*}
[u^{\mathsf T}X_t]
-
m_u
&=
u^{\mathsf T}
e^{tA}
\left(
n_0^*-\mathbb E_\pi X
\right)
=
\Delta e^{-\rho t}.
\end{aligned}
\end{equation*}
The function $f(n):=u^{\mathsf T}n/\|u\|_\infty$ is $1$-Lipschitz with respect to the $\ell^1$ metric. Kantorovich-Rubinstein duality \eqref{eq:KR_duality} therefore gives
\begin{equation*}
\begin{aligned}
W_1
\bigl(
P_t(n_0^*,\cdot),\pi
\bigr)
&\ge
\frac{
\left|
\mathbb E_{n_0^*}
[u^{\mathsf T}X_t]
-
m_u
\right|
}{
\|u\|_\infty
}
=
\frac{\Delta}{\|u\|_\infty}
e^{-\rho t},
\qquad t\ge0.
\end{aligned}
\end{equation*}
This proves part~(ii) and completes the proof.
\end{proof}

\section{Proof of Theorem~\ref{thm:critical_regular}}\label{sec:critical_regular_proof}
Throughout this section, fix $n_0\in E$ and assume that $\lambda_{\max}=0$, conditions (C1)-(C2) hold, and the regularity condition is satisfied for $n_0$. The main difference from the stable case is the presence of conserved zero modes. These modes generate a finite ``source'' component, while all remaining persistent species form an exponentially stable downstream subsystem. The proof is based on this decomposition.

As noted in Remark~\ref{rem:basis_independence}, the partition into sets $\mathcal B_c$ is independent of the particular non-negative basis chosen for $\ker(A^{\mathsf T})$, up to a relabeling of $c$. We thus use throughout this section the adapted non-negative basis $\{\alpha^{(1)},\dots,\alpha^{(m)}\}$ constructed in Supplementary Lemma~S3.2. Define
\begin{equation*}
K(n_0)
:=
\left\{
k:
\alpha^{(k)}\cdot n_0>0
\right\},
\end{equation*}
and, for $k\in K(n_0)$,
\begin{equation*}
Z_k
:=
\left\{
i\in I_+:
\alpha_i^{(k)}>0
\right\}.
\end{equation*}
Set
\begin{equation*}
Z
:=
\bigcup_{k\in K(n_0)}Z_k,
\qquad
D
:=
I_+\setminus Z.
\end{equation*}
We refer to $Z_k$ as the $k$th source block, to $Z$ as the source set, and to $D$ as the downstream set. The structural results in Supplementary Section~S3 show that the source blocks are nonempty and pairwise disjoint and coincide with the activated zero-eigenvalue strongly connected components. Moreover, $\alpha^{(k)}$ is constant and strictly positive on $Z_k$. By rescaling the basis vector, we may normalize this constant to one, so that, on $I_+$,
\begin{equation*}
\alpha_i^{(k)}
=
\mathbf1_{\{i\in Z_k\}},
\qquad
k\in K(n_0).
\end{equation*}
The regularity condition then implies that every reaction of the restricted network preserves the total number of molecules in each source block $Z_k$. In particular, a first-order reaction with reactant in $Z_k$ replaces the reactant by exactly one molecule in the same source block and may additionally produce molecules in $D$. No zero-order reaction and no reaction with reactant in $D$ can produce molecules in $Z$. Thus the source coordinates evolve autonomously and drive, but are not influenced by, the downstream coordinates.

For every $c$ such that $\mathcal B_c\ne\varnothing$ and every $n\in\mathcal B_c$,
\begin{equation*}
\sum_{i\in Z_k}n_i
=
\alpha^{(k)}\cdot n
=
c_k,
\qquad
k\in K(n_0),
\end{equation*}
and hence
\begin{equation*}
\|n_Z\|_1
=
\sum_{k\in K(n_0)}c_k
=:
L_c.
\end{equation*}
Thus $c_k$ is the conserved number of source molecules in $Z_k$, while $L_c$ is the total source population on $\mathcal B_c$. The possible source configurations therefore form the finite set
\begin{equation*}
\mathcal Z_c
:=
\left\{
z\in\mathbb N^{|Z|}:
\sum_{i\in Z_k}z_i=c_k,
\quad
k\in K(n_0)
\right\},
\end{equation*}
with the convention $\mathcal Z_c=\{\mathbf0\}$ when $K(n_0)=\varnothing$.

We first establish the moment bounds needed below.

\begin{proposition}\label{prop:lyapunov_moments_crit}
Let $c\in\mathbb R_{\ge0}^m$ satisfy $\mathcal B_c\ne\varnothing$. Then $\mathcal B_c$ is closed, and the process restricted to the set $\mathcal B_c$ admits a stationary distribution $\pi_c$. Moreover, there exists $\theta>0$ such that, for every $n\in\mathcal B_c$,
\begin{equation}\label{eq:exponential moment}
\sup_{t\ge0}
\mathbb E_n
\left[
e^{\theta\|X_t\|_1}
\right]
<
\infty,
\qquad
\mathbb E_{\pi_c}
\left[
e^{\theta\|X\|_1}
\right]
<
\infty.
\end{equation}
In particular,
\begin{equation*}
\sup_{t\ge0}\mathbb E_n\|X_t\|_1^p<\infty,
\qquad
\mathbb E_{\pi_c}\|X\|_1^p<\infty
\end{equation*}
for every integer $p\ge1$.
\end{proposition}

\begin{proof}
The proof is given in Supplementary Section~S3.
\end{proof}

We next establish a contraction estimate on a fixed nonempty set $\mathcal B_c$. Let $P_t^c$ denote the semigroup of the process restricted to the set $\mathcal B_c$, and write $n=(n_Z,n_D)\in\mathcal Z_c\times\mathbb N^{|D|}$ for $n\in\mathcal B_c$. Fix $\varepsilon\in(0,\lambda_{\mathrm{gap}})$. If $D\ne\varnothing$, order the species as $I_0,Z,D$. The closedness of $\mathcal B$, together with the one-way structure from $Z$ to $D$, gives
\begin{equation*}
A
=
\begin{pmatrix}
A_{I_0}&0&0\\
*&A_Z&0\\
*&*&A_D
\end{pmatrix}.
\end{equation*}
All zero eigenmodes supported in $I_+$ are carried by the source blocks in $Z$. Hence $0\notin\sigma(A_D)$. Since $A_D$ is a diagonal block of the above block-triangular representation, $\sigma(A_D)\subseteq\sigma(A)$ and therefore
\begin{equation*}
\lambda_{\max}(A_D)
\le
-\lambda_{\mathrm{gap}}.
\end{equation*}
Therefore the downstream subsystem is exponentially stable. We now construct a coupling of two copies $X_t=(Z_t,D_t)$ and $Y_t=(Z_t',D_t')$, starting from $n=(n_Z,n_D)$ and $n'=(n_Z',n_D')$ in the same set $\mathcal B_c$.

We first construct a coupling for the source components. For every $k\in K(n_0)$, both source configurations contain exactly $c_k$ particles in $Z_k$. Label these particles by $1,\dots,c_k$ in each copy and pair particles having the same label, so that there are altogether $L_c$ labelled pairs. Consider the motion of a single labelled particle in $Z_k$. For each reaction $R_{ij}$ with $i\in Z_k$, regularity implies that there exists a unique $\chi_{ij}\in Z_k$ such that
\begin{equation*}
R_{ij}:
\qquad
S_i
\longrightarrow
S_{\chi_{ij}}
+
\sum_{h\in D}
\nu_h^{ij}S_h.
\end{equation*}
The source position of the labelled particle hence evolves as a finite irreducible continuous-time Markov chain $(\Xi_t^{(k)})_{t\ge0}$ on $Z_k$. We retain the full reaction-channel marks, including reaction events that leave the source position unchanged but produce downstream molecules, and denote the transition probabilities of the source-position chain by
\begin{equation*}
p_t^{(k)}(i,h)
:=
\mathbb P
\left(
\Xi_t^{(k)}=h
\mid
\Xi_0^{(k)}=i
\right).
\end{equation*}
We couple each labelled pair on successive time intervals $[0,T],[T,2T],\dots$. If the two particles occupy the same source species at a grid time, they are coupled synchronously from that time onward, using the same reaction clocks and channel marks. If instead they occupy distinct states $i,i'\in Z_k$ at time $rT$, then at time $(r+1)T$ their marginal distributions are $p_T^{(k)}(i,\cdot)$ and $p_T^{(k)}(i',\cdot)$, which we couple maximally. Thus there exists a coupling $\mathsf M_{i,i'}^{(k)}$ satisfying
\begin{equation*}
\sum_{h'\in Z_k}
\mathsf M_{i,i'}^{(k)}(h,h')
=
p_T^{(k)}(i,h),
\qquad
\sum_{h\in Z_k}
\mathsf M_{i,i'}^{(k)}(h,h')
=
p_T^{(k)}(i',h'),
\end{equation*}
and
\begin{equation*}
\sum_{h\in Z_k}
\mathsf M_{i,i'}^{(k)}(h,h)
=
1-
\left\|
p_T^{(k)}(i,\cdot)
-
p_T^{(k)}(i',\cdot)
\right\|_{\mathrm{TV}}.
\end{equation*}
For each sampled endpoint pair, the paths over the grid interval are completed using the corresponding marked bridge laws, which preserves the complete marked source-process law in each marginal. The nonzero eigenvalues of the source-position generators are nonzero eigenvalues of $A$, so their spectral gaps are at least $\lambda_{\mathrm{gap}}$. Since there are only finitely many source blocks, finite-state spectral theory allows us to choose $T>0$ so that the probability that an uncoupled labelled pair remains uncoupled contracts uniformly at any prescribed exponential rate below $\lambda_{\mathrm{gap}}$; below we use the rate $\lambda_{\mathrm{gap}}-\varepsilon$. Different labelled pairs are coupled independently until they meet, after which the sticky synchronous coupling described above is used.

We next construct a coupling for the downstream components. Condition on the complete marked source paths. Every source reaction and its downstream product vector are then fixed. Whenever a marked source reaction occurs in one coordinate, its downstream products are added to the corresponding downstream component. Once a source pair has coupled, its marked source paths coincide, so all subsequent downstream products generated by that pair are added simultaneously to both downstream components. Between source-reaction marks, we couple the downstream dynamics exactly as in the stable case. At a state $(d,d')\in\mathbb N^{|D|}\times\mathbb N^{|D|}$, every zero-order reaction $R_{0j}$ induces
\begin{equation*}
(d,d')
\longrightarrow
\bigl(
d+\nu_{0j}|_D,
d'+\nu_{0j}|_D
\bigr)
\qquad
\text{at rate }k_{0j},
\end{equation*}
and every first-order reaction $R_{ij}$ with $i\in D$ induces
\begin{equation*}
\begin{array}{lll}
(d,d')
\longrightarrow
\bigl(
d+w_{ij}|_D,
d'+w_{ij}|_D
\bigr)
&\text{at rate}&
k_{ij}\min\{d_i,d_i'\},
\\[1mm]
(d,d')
\longrightarrow
\bigl(
d+w_{ij}|_D,
d'
\bigr)
&\text{at rate}&
k_{ij}(d_i-d_i')_+,
\\[1mm]
(d,d')
\longrightarrow
\bigl(
d,
d'+w_{ij}|_D
\bigr)
&\text{at rate}&
k_{ij}(d_i'-d_i)_+.
\end{array}
\end{equation*}
The marked source construction preserves the correct source dynamics in each marginal, while the above rates sum to the correct downstream reaction rates. Hence both marginals are distributed according to the original process restricted to the set $\mathcal B_c$. Non-explosivity follows from non-explosivity of the marginals. When $D=\varnothing$, only the source coupling is required.

The resulting coupling yields the following estimate.

\begin{lemma}\label{lem:critical_semigroup_comparison}
For the coupling constructed above, there exists $C_W(\varepsilon)<\infty$, independent of $c$, such that
\begin{equation}\label{eq:W1}
W_1
\bigl(
P_t^c(n,\cdot),
P_t^c(n',\cdot)
\bigr)
\le
C_W(\varepsilon)
\bigl(
1+\|n\|_1+\|n'\|_1
\bigr)
e^{-(\lambda_{\mathrm{gap}}-\varepsilon)t},
\qquad t\ge0.
\end{equation}
\end{lemma}

\begin{proof}
The proof is given in Supplementary Section~S3.
\end{proof}

We next pass from pairwise contraction to convergence toward stationarity.

\begin{proposition}\label{prop:W1_upper_crit}
Every nonempty $\mathcal B_c$ admits a unique stationary distribution $\pi_c = (\pi_c(n))_{n\in\mathcal B_c}$. Moreover, there exists a constant $C_0<\infty$, independent of $c$, such that
\begin{equation}\label{eq:stationary_class_moment_crit}
m_c
:=
\sum_{y\in\mathcal B_c}
\|y\|_1\pi_c(y)
\le
C_0(1+L_c).
\end{equation}
For every $\varepsilon\in(0,\lambda_{\mathrm{gap}})$, there exist constants $C_1(\varepsilon),C_2(\varepsilon)<\infty$, independent of $c$, such that
\begin{equation}\label{eq:C_def_critical}
W_1
\bigl(
P_t^c(x,\cdot),\pi_c
\bigr)
\le
\left[
C_1(\varepsilon)\|x\|_1+
C_2(\varepsilon)
\right]
e^{-(\lambda_{\mathrm{gap}}-\varepsilon)t}
\end{equation}
for every $x\in\mathcal B_c$ and $t\ge0$.
\end{proposition}

\begin{proof}
The proof is given in Supplementary Section~S3.
\end{proof}

The preceding proposition controls the dynamics after the process has entered $\mathcal B$. To extend the estimate to the original initial state $n_0$, we use the following weighted hitting-time estimate, which requires only conditions (C1)-(C2) and not regularity.

\begin{lemma}\label{lem:weighted_exp_moment_crit}
Let $\tau_{\mathcal B}$ denote the hitting time of $\mathcal B$. For every $\varepsilon\in(0,\lambda_{\mathrm{gap}})$,
\begin{equation*}
\mathbb E_{n_0}
\left[
(1+\|X_{\tau_{\mathcal B}}\|_1)
e^{(\lambda_{\mathrm{gap}}-\varepsilon)
\tau_{\mathcal B}}
\right]
+
\sup_{t\ge0}
e^{(\lambda_{\mathrm{gap}}-\varepsilon)t}
\mathbb E_{n_0}
\left[
\mathbf1_{\{\tau_{\mathcal B}>t\}}
(1+\|X_t\|_1)
\right]
<
\infty.
\end{equation*}
\end{lemma}

\begin{proof}
The proof is given in Supplementary Section~S3.
\end{proof}

We are now in a position to prove Theorem~\ref{thm:critical_regular}.

\begin{proof}[Proof of Theorem~\ref{thm:critical_regular}]
By Proposition~\ref{prop:W1_upper_crit}, every nonempty $\mathcal B_c$ admits a unique stationary distribution $\pi_c$, while Lemma~\ref{lem:weighted_exp_moment_crit} implies that $\tau_{\mathcal B}<\infty$ almost surely. Define
\begin{equation*}
c(X_{\tau_{\mathcal B}})
:=
\bigl(
\alpha^{(1)}\cdot X_{\tau_{\mathcal B}},
\dots,
\alpha^{(m)}\cdot X_{\tau_{\mathcal B}}
\bigr)
\end{equation*}
and
\begin{equation*}
\pi_{n_0}
:=
\mathbb E_{n_0}
\left[
\pi_{c(X_{\tau_{\mathcal B}})}
\right]
=
\sum_{\mathcal B_c\ne\varnothing}
\mathbb P_{n_0}
\left(
c(X_{\tau_{\mathcal B}})=c
\right)
\pi_c.
\end{equation*}
Therefore the mixing coefficients are precisely the probabilities with which the process enters the corresponding closed sets $\mathcal B_c$. Since
\begin{equation*}
L_{c(X_{\tau_{\mathcal B}})}
=
\|(X_{\tau_{\mathcal B}})_Z\|_1
\le
\|X_{\tau_{\mathcal B}}\|_1,
\end{equation*}
it follows from \eqref{eq:stationary_class_moment_crit} that
\begin{equation}\label{eq:mc_bound}
m_{c(X_{\tau_{\mathcal B}})}
\le
C_0
\left(
1+\|X_{\tau_{\mathcal B}}\|_1
\right).
\end{equation}
Together with Lemma~\ref{lem:weighted_exp_moment_crit}, this yields
\begin{equation*}
\sum_{y\in E}
\|y\|_1\pi_{n_0}(y)
=
\mathbb E_{n_0}
\left[
m_{c(X_{\tau_{\mathcal B}})}
\right]
<
\infty.
\end{equation*}

We now use the strong Markov property at the hitting time $\tau_{\mathcal B}$ to combine the classwise estimates. On $\{\tau_{\mathcal B}\le t\}$, condition on $\bigl(\tau_{\mathcal B},X_{\tau_{\mathcal B}},c(X_{\tau_{\mathcal B}})\bigr)$ and couple $P_{t-\tau_{\mathcal B}}^{c(X_{\tau_{\mathcal B}})}(X_{\tau_{\mathcal B}},\cdot)$ with $\pi_{c(X_{\tau_{\mathcal B}})}$. On $\{\tau_{\mathcal B}>t\}$, condition on $c(X_{\tau_{\mathcal B}})$ and couple $X_t$ with an independent sample from $\pi_{c(X_{\tau_{\mathcal B}})}$. This construction is only needed at the level of terminal distributions and need not be adapted to the original filtration. By the coupling characterization of $W_1$ and convexity under mixtures,
\begin{equation*}
\begin{aligned}
W_1
\bigl(
P_t(n_0,\cdot),\pi_{n_0}
\bigr)
\le{}&
\mathbb E_{n_0}
\Bigg[
\mathbf1_{\{\tau_{\mathcal B}\le t\}}
W_1
\left(
P_{t-\tau_{\mathcal B}}^{c(X_{\tau_{\mathcal B}})}
(X_{\tau_{\mathcal B}},\cdot),
\pi_{c(X_{\tau_{\mathcal B}})}
\right)
\Bigg]
\\
&+
\mathbb E_{n_0}
\left[
\mathbf1_{\{\tau_{\mathcal B}>t\}}
\left(
\|X_t\|_1+
m_{c(X_{\tau_{\mathcal B}})}
\right)
\right].
\end{aligned}
\end{equation*}
For the first term, Proposition~\ref{prop:W1_upper_crit} and Lemma~\ref{lem:weighted_exp_moment_crit} give
\begin{equation*}
\begin{aligned}
&\mathbb E_{n_0}
\Bigg[
\mathbf1_{\{\tau_{\mathcal B}\le t\}}
W_1
\left(
P_{t-\tau_{\mathcal B}}^{c(X_{\tau_{\mathcal B}})}
(X_{\tau_{\mathcal B}},\cdot),
\pi_{c(X_{\tau_{\mathcal B}})}
\right)
\Bigg]
\\
&\qquad\le
e^{-(\lambda_{\mathrm{gap}}-\varepsilon)t}
\mathbb E_{n_0}
\left[
\left(
C_1(\varepsilon)\|X_{\tau_{\mathcal B}}\|_1+
C_2(\varepsilon)
\right)
e^{(\lambda_{\mathrm{gap}}-\varepsilon)\tau_{\mathcal B}}
\right]
\\
&\qquad\le
C(n_0,\varepsilon)
e^{-(\lambda_{\mathrm{gap}}-\varepsilon)t}.
\end{aligned}
\end{equation*}
For the second term, \eqref{eq:mc_bound} yields
\begin{equation*}
\begin{aligned}
&
\mathbb E_{n_0}
\left[
\mathbf1_{\{\tau_{\mathcal B}>t\}}
\left(
\|X_t\|_1+
m_{c(X_{\tau_{\mathcal B}})}
\right)
\right]
\\
&\qquad\le
C\,
\mathbb E_{n_0}
\left[
\mathbf1_{\{\tau_{\mathcal B}>t\}}
(1+\|X_t\|_1)
\right]
+
C\,
\mathbb E_{n_0}
\left[
\mathbf1_{\{\tau_{\mathcal B}>t\}}
(1+\|X_{\tau_{\mathcal B}}\|_1)
\right].
\end{aligned}
\end{equation*}
The first expectation is bounded by $C(n_0,\varepsilon)e^{-(\lambda_{\mathrm{gap}}-\varepsilon)t}$ by Lemma~\ref{lem:weighted_exp_moment_crit}. Moreover, on $\{\tau_{\mathcal B}>t\}$,
\begin{equation*}
e^{(\lambda_{\mathrm{gap}}-\varepsilon)t}
\le
e^{(\lambda_{\mathrm{gap}}-\varepsilon)\tau_{\mathcal B}},
\end{equation*}
so that
\begin{equation*}
\begin{aligned}
&
\mathbb E_{n_0}
\left[
\mathbf1_{\{\tau_{\mathcal B}>t\}}
(1+\|X_{\tau_{\mathcal B}}\|_1)
\right]
\\
&\qquad\le
e^{-(\lambda_{\mathrm{gap}}-\varepsilon)t}
\mathbb E_{n_0}
\left[
(1+\|X_{\tau_{\mathcal B}}\|_1)
e^{(\lambda_{\mathrm{gap}}-\varepsilon)\tau_{\mathcal B}}
\right]
\\
&\qquad\le
C(n_0,\varepsilon)
e^{-(\lambda_{\mathrm{gap}}-\varepsilon)t}.
\end{aligned}
\end{equation*}
Combining the above estimates gives
\begin{equation*}
W_1
\bigl(
P_t(n_0,\cdot),\pi_{n_0}
\bigr)
\le
C(n_0,\varepsilon)
e^{-(\lambda_{\mathrm{gap}}-\varepsilon)t},
\qquad t\ge0.
\end{equation*}
Each $\pi_c$ is stationary on the closed set $\mathcal B_c$, so their mixture $\pi_{n_0}$ is stationary. Finally, since the $\ell^1$ distance between two distinct states in $E=\mathbb N^N$ is at least one,
\begin{equation*}
\left\|
P_t(n_0,\cdot)-\pi_{n_0}
\right\|_{\mathrm{TV}}
\le
W_1
\bigl(
P_t(n_0,\cdot),\pi_{n_0}
\bigr).
\end{equation*}
This proves the asserted exponential convergence in both distances.
\end{proof}

\section{Proof of Theorem~\ref{thm:critical_irregular}}\label{sec:critical_irregular_proof}
Throughout this section, assume that $\lambda_{\max}=0$ and that conditions (C1)-(C2) hold. Fix $n_0\in E$ and suppose that the regularity condition is violated for this initial state. We again use the adapted non-negative basis $\{\alpha^{(1)},\dots,\alpha^{(m)}\}$ introduced in Section~\ref{sec:critical_regular_proof}. By the characterization proved in Supplementary Section~S3, failure of regularity is detected by at least one basis vector. Define
\begin{equation*}
K_{\mathrm{irr}}
:=
\left\{
k:
\alpha^{(k)}|_{I_+}
\cdot
(w_R|_{I_+})
\ne0
\text{ for some reaction }R
\text{ of the restricted network}
\right\}.
\end{equation*}
Then $K_{\mathrm{irr}}\ne\varnothing$. After relabeling the basis vectors, we may write $K_{\mathrm{irr}}=\{1,\dots,l\}$ for some $1\le l\le m$, and we refer to $\alpha^{(1)},\dots,\alpha^{(l)}$ as the irregular basis vectors. The key mechanism is that each irregular basis vector generates a non-negative martingale that eventually becomes extinct almost surely, but with a heavy-tailed extinction time. Once all irregular modes have disappeared, the process reduces to a stable or regular critical subsystem. This yields convergence in total variation, whereas conservation of the martingale means prevents convergence in $W_1$ and the heavy extinction tail rules out exponential convergence in total variation.

We first record the martingale property.

\begin{proposition}\label{prop:martingale}
Let $\alpha\in\ker(A^{\mathsf T})$ be non-negative and satisfy $\alpha\cdot b=0$. Then for every initial state $n_0\in E$, $M_t:=\alpha\cdot X_t$ is a non-negative martingale satisfying
\begin{equation*}
\mathbb E_{n_0}M_t
=
M_0
=
\alpha\cdot n_0,
\qquad t\ge0.
\end{equation*}
If $\alpha$ is a conservation law, then $M_t=M_0$ almost surely for every $t\ge0$.
\end{proposition}

\begin{proof}
Since $\alpha\in\ker(A^{\mathsf T})$ and $\alpha\cdot b=0$,
\begin{equation*}
Q(\alpha\cdot n)
=
\alpha\cdot(An+b)
=
0,
\qquad n\in E.
\end{equation*}
Hence $M_t$ is a local martingale \cite{norris1998markov}. Since $M_t\ge0$, it is also a supermartingale. On the other hand, the mean equation gives
\begin{equation*}
\begin{aligned}
\mathbb E_{n_0}M_t
&=
\alpha^{\mathsf T}\mathbb E_{n_0}X_t
\\
&=
\alpha^{\mathsf T}n_0
+
\int_0^t
\alpha^{\mathsf T}
\left(
A\mathbb E_{n_0}X_s+b
\right)
\,\mathrm ds
=
\alpha^{\mathsf T}n_0
=
M_0.
\end{aligned}
\end{equation*}
Thus the non-negative supermartingale $M_t$ has constant expectation and is therefore a martingale. If $\alpha$ is a conservation law, then $\alpha\cdot w_R=0$ for every reaction $R$, so every jump leaves $\alpha\cdot X_t$ unchanged and hence $M_t=M_0$ almost surely for all $t\ge0$.
\end{proof}

For the irregular basis vectors, define their joint extinction time by
\begin{equation*}
\tau_{\mathrm{irr}}
:=
\inf
\left\{
t\ge0:
\alpha^{(i)}\cdot X_t=0
\text{ for all }
i=1,\dots,l
\right\}.
\end{equation*}

\begin{proposition}\label{prop:joint_irregular_extinction}
For every $i=1,\dots,l$, we have $\alpha^{(i)}\cdot n_0>0$. Moreover, $\tau_{\mathrm{irr}}<\infty$ almost surely, and there exists a constant $c_0>0$ such that
\begin{equation}\label{eq:tau_irr}
\mathbb P_{n_0}
(\tau_{\mathrm{irr}}>t)
\ge
\frac{c_0}{1+t},
\qquad t\ge0.
\end{equation}
Consequently, $\mathbb E_{n_0}\tau_{\mathrm{irr}}=\infty$, and $\tau_{\mathrm{irr}}$ admits no finite positive exponential moment, i.e.
\[
\mathbb E_{n_0}
\left[
e^{\lambda\tau_{\mathrm{irr}}}
\right]
=
\infty
\qquad
\text{for every } \lambda>0.
\]
\end{proposition}

\begin{proof}
The proof is given in Supplementary Section~S3.
\end{proof}

Thus all irregular modes disappear in finite time almost surely, although their joint extinction time has a polynomially decaying tail. To describe the dynamics after extinction, set
\begin{equation*}
U_{\mathrm{irr}}
:=
\bigcup_{i=1}^l
\operatorname{supp}\alpha^{(i)},
\qquad
E_0
:=
\left\{
n\in E:
\alpha^{(i)}\cdot n=0,
\quad
i=1,\dots,l
\right\},
\end{equation*}
and define $J:=I_+\setminus U_{\mathrm{irr}}$. Since each $\alpha^{(i)}$ is strictly positive on its support,
\begin{equation*}
E_0
=
\left\{
n\in E:
n_{U_{\mathrm{irr}}}=0
\right\}.
\end{equation*}
Hence, on $\mathcal B\cap E_0$, all species in $I_0\cup U_{\mathrm{irr}}$ are absent, and only the species in $J$ remain dynamically relevant.

\begin{proposition}\label{prop:irregular_reduced_dynamics}
The set $\mathcal B\cap E_0$ is closed. The projection of the process on $\mathcal B\cap E_0$ onto the species in $J$ evolves autonomously as a linear reaction network. The reduced network is either stable or satisfies conditions (C1)-(C2) in the critical regime; in the latter case, every left null vector of its first-order matrix is a conservation law.

For $c_{\mathrm{irr}}=(c_{l+1},\dots,c_m)$, define
\begin{equation*}
\mathcal B'_{c_{\mathrm{irr}}}
:=
\left\{
n\in\mathcal B\cap E_0:
\alpha^{(k)}\cdot n=c_k,
\quad
k=l+1,\dots,m
\right\},
\end{equation*}
with the empty label understood when $l=m$. Every nonempty $\mathcal B'_{c_{\mathrm{irr}}}$ is a closed set and admits a unique stationary distribution $\pi_{c_{\mathrm{irr}}}$. Moreover, for every $\varepsilon\in(0,\lambda_{\mathrm{gap}})$ and $x\in\mathcal B'_{c_{\mathrm{irr}}}$, there exists $C(x,\varepsilon)<\infty$ such that
\begin{equation*}
\left\|
P_t(x,\cdot)
-
\pi_{c_{\mathrm{irr}}}
\right\|_{\mathrm{TV}}
\le
W_1
\bigl(
P_t(x,\cdot),
\pi_{c_{\mathrm{irr}}}
\bigr)
\le
C(x,\varepsilon)
e^{-(\lambda_{\mathrm{gap}}-\varepsilon)t},
\qquad t\ge0.
\end{equation*}
There also exists a constant $C<\infty$, independent of $c_{\mathrm{irr}}$, such that
\begin{equation}\label{eq:irregular_stationary_moment}
\sum_{n\in E}
\|n\|_1
\pi_{c_{\mathrm{irr}}}(n)
\le
C
\left(
1+
\sum_{k=l+1}^m c_k
\right).
\end{equation}
\end{proposition}

\begin{proof}
The proof is given in Supplementary Section~S3.
\end{proof}

We are now in a position to prove Theorem~\ref{thm:critical_irregular}.

\begin{proof}[Proof of Theorem~\ref{thm:critical_irregular}]
Let $\tau_{\mathrm{irr}}$ be defined as above and set $\tau:=\tau_{\mathrm{irr}}\vee\tau_{\mathcal B}$. By Theorem~\ref{thm:hitting_critical} and Proposition~\ref{prop:joint_irregular_extinction}, $\tau<\infty$ almost surely. Since both $\mathcal B$ and $E_0$ are closed, $X_t\in\mathcal B\cap E_0$ for every $t\ge\tau$. Define
\begin{equation*}
c_{\mathrm{irr}}(X_\tau)
:=
\bigl(
\alpha^{(l+1)}\cdot X_\tau,
\dots,
\alpha^{(m)}\cdot X_\tau
\bigr),
\end{equation*}
with the empty label understood when $l=m$. Then $X_\tau\in\mathcal B'_{c_{\mathrm{irr}}(X_\tau)}$, and Proposition~\ref{prop:irregular_reduced_dynamics} implies that the process remains in this closed set thereafter. Define
\begin{equation*}
\pi_{n_0}
:=
\sum_{c_{\mathrm{irr}}}
\mathbb P_{n_0}
\left(
c_{\mathrm{irr}}(X_\tau)
=
c_{\mathrm{irr}}
\right)
\pi_{c_{\mathrm{irr}}},
\end{equation*}
where the sum extends over the values attained by $c_{\mathrm{irr}}(X_\tau)$. Since every $\mathcal B'_{c_{\mathrm{irr}}}$ is closed, each $\pi_{c_{\mathrm{irr}}}$ is stationary for the original process. Hence $\pi_{n_0}$ is stationary and supported on $\mathcal B\cap E_0$.

We first prove convergence in total variation. Conditioning on $\bigl(\tau,X_\tau,c_{\mathrm{irr}}(X_\tau)\bigr)$ and using the strong Markov property together with convexity of total variation under mixtures give
\begin{equation*}
\begin{aligned}
\left\|
P_t(n_0,\cdot)-\pi_{n_0}
\right\|_{\mathrm{TV}}
\le
\mathbb E_{n_0}
\Bigg[
&
\mathbf1_{\{\tau\le t\}}
\left\|
P_{t-\tau}(X_\tau,\cdot)
-
\pi_{c_{\mathrm{irr}}(X_\tau)}
\right\|_{\mathrm{TV}}
+
\mathbf1_{\{\tau>t\}}
\Bigg].
\end{aligned}
\end{equation*}
Since $\tau<\infty$ almost surely, Proposition~\ref{prop:irregular_reduced_dynamics} implies
\begin{equation*}
\left\|
P_{t-\tau}(X_\tau,\cdot)
-
\pi_{c_{\mathrm{irr}}(X_\tau)}
\right\|_{\mathrm{TV}}
\longrightarrow0
\qquad
\text{as }t\to\infty
\end{equation*}
for almost every sample path. The integrand is bounded by one, so dominated convergence gives
\begin{equation*}
\left\|
P_t(n_0,\cdot)-\pi_{n_0}
\right\|_{\mathrm{TV}}
\longrightarrow0,
\qquad
t\to\infty.
\end{equation*}

It remains to show that convergence fails in $W_1$ and that the total variation convergence cannot be exponential. The mean equation gives $\mathbb E_{n_0}\|X_t\|_1<\infty$ for every finite $t$. We first verify that $\pi_{n_0}$ also has a finite first moment. By \eqref{eq:irregular_stationary_moment},
\begin{equation*}
\begin{aligned}
\sum_{n\in E}
\|n\|_1
\pi_{n_0}(n)
&=
\mathbb E_{n_0}
\left[
\sum_{n\in E}
\|n\|_1
\pi_{c_{\mathrm{irr}}(X_\tau)}(n)
\right]
\\
&\le
C
\left(
1+
\sum_{k=l+1}^m
\mathbb E_{n_0}
\left[
\alpha^{(k)}\cdot X_\tau
\right]
\right).
\end{aligned}
\end{equation*}
For every $k=l+1,\dots,m$, Proposition~\ref{prop:martingale} shows that $\alpha^{(k)}\cdot X_t$ is a non-negative martingale. Applying the optional sampling theorem at $\tau\wedge r$ gives
\begin{equation*}
\mathbb E_{n_0}
\left[
\alpha^{(k)}\cdot X_{\tau\wedge r}
\right]
=
\alpha^{(k)}\cdot n_0.
\end{equation*}
Since $\tau<\infty$ almost surely, Fatou's lemma yields
\begin{equation*}
\mathbb E_{n_0}
\left[
\alpha^{(k)}\cdot X_\tau
\right]
\le
\alpha^{(k)}\cdot n_0.
\end{equation*}
Consequently,
\begin{equation*}
\sum_{n\in E}
\|n\|_1
\pi_{n_0}(n)
\le
C
\left(
1+
\sum_{k=l+1}^m
\alpha^{(k)}\cdot n_0
\right)
<
\infty,
\end{equation*}
so $W_1(P_t(n_0,\cdot),\pi_{n_0})$ is well defined.

Fix $\alpha:=\alpha^{(1)}$. By Proposition~\ref{prop:joint_irregular_extinction}, $\alpha\cdot n_0>0$. Since $\pi_{n_0}$ is supported on $E_0$, $\alpha\cdot y=0$ for $\pi_{n_0}$-almost every $y$. Therefore, by Kantorovich-Rubinstein duality applied to the $1$-Lipschitz function $f(n):=\alpha\cdot n/\|\alpha\|_\infty$,
\begin{equation*}
\begin{aligned}
W_1
\bigl(
P_t(n_0,\cdot),\pi_{n_0}
\bigr)
&\ge
\frac{
\left|
\mathbb E_{n_0}
[\alpha\cdot X_t]
-
\mathbb E_{\pi_{n_0}}
[\alpha\cdot X]
\right|
}{
\|\alpha\|_\infty
}
\\
&=
\frac{
\alpha\cdot n_0
}{
\|\alpha\|_\infty
}
>
0,
\qquad t\ge0,
\end{aligned}
\end{equation*}
where we used Proposition~\ref{prop:martingale} in the last equality. Hence convergence does not hold in $W_1$.

Finally, since $E_0$ is closed and
\begin{equation*}
\tau_{\mathrm{irr}}
=
\inf\{t\ge0:X_t\in E_0\},
\end{equation*}
we have
\begin{equation*}
P_t
\bigl(
n_0,E\setminus E_0
\bigr)
=
\mathbb P_{n_0}
(\tau_{\mathrm{irr}}>t),
\end{equation*}
whereas $\pi_{n_0}(E\setminus E_0)=0$. Therefore, by \eqref{eq:tau_irr},
\begin{equation*}
\begin{aligned}
\left\|
P_t(n_0,\cdot)-\pi_{n_0}
\right\|_{\mathrm{TV}}
&\ge
P_t
\bigl(
n_0,E\setminus E_0
\bigr)
\\
&=
\mathbb P_{n_0}
(\tau_{\mathrm{irr}}>t)
\ge
\frac{c_0}{1+t},
\qquad t\ge0.
\end{aligned}
\end{equation*}
Hence the convergence in total variation cannot be exponential. Together with the failure of $W_1$ convergence proved above, this completes the proof.
\end{proof}

\section{Conclusions and discussion}\label{sec:conclusion}
In the present paper, we characterize the long-time behavior of reducible stochastic linear reaction networks under both the $L^1$-Wasserstein and total variation distances, without requiring irreducibility or any additional structural assumptions. Our results show that the asymptotic behavior of the stochastic dynamics is governed by three structural features of the underlying reaction network: (i) the spectrum of the first-order influx matrix $A$, (ii) the position of the zero-order influx vector $b$ relative to the left nullspace of $A$, and (iii) whether every left null vector of $A$ corresponds to a conservation law when the reaction network is restricted to the persistent species. Based on these features, we identify the conditions under which the time-dependent distribution converges exponentially fast to a stationary limiting distribution, converges only at a non-exponential rate, or fails to exhibit positive recurrence in the interior of the state space. Furthermore, we prove that the stochastic dynamics of every linear reaction network is non-explosive.

In the stable regime $\lambda_{\max}<0$, the deterministic and stochastic dynamics exhibit closely parallel long-time behaviors. The deterministic system converges exponentially fast to its globally attracting equilibrium, whereas the stochastic system enters an absorbing set in finite time. This absorbing set is the unique closed irreducible positive recurrent class and supports the unique stationary distribution of the system. By constructing suitable couplings, we prove that the stochastic process converges exponentially fast to its stationary distribution under both the $L^1$-Wasserstein and total variation distances, with the optimal exponential convergence rate given by $-\lambda_{\max}$. Therefore, the dominant spectral mode of the deterministic system also determines the asymptotic relaxation rate of the stochastic dynamics.

The critical regime $\lambda_{\max}=0$ exhibits a substantially richer structure. When the zero eigenvalue of $A$ is semisimple and $b$ is orthogonal to the left nullspace of $A$, the deterministic system converges exponentially fast to an equilibrium that may depend on the initial condition. In contrast, for the stochastic system, these two spectral conditions alone are not sufficient to ensure exponential convergence. An additional decisive requirement is the regularity condition introduced in Definition~\ref{regular}, which requires that the left null vectors of $A$ correspond to conservation laws of the reaction network restricted to the persistent species. When the regularity condition is satisfied, the absorbing set decomposes
into multiple closed classes, each of which supports a stationary distribution. By constructing suitable couplings, we prove that the stochastic process converges exponentially fast to an initial-state-dependent mixture of these stationary distributions.
When the regularity condition holds for every initial state, the spectral gap $\lambda_{\rm gap}$ is the optimal uniform exponential convergence rate. If regularity fails for some initial state, exponential convergence fails for that state.

When the regularity condition is violated, the distinction between the total variation and $L^1$-Wasserstein distances becomes essential. Using martingale convergence arguments, we prove that the stochastic process still converges in total variation to a stationary limiting distribution, although this convergence is no longer exponential. In contrast, convergence in the $L^1$-Wasserstein distance fails in general. This phenomenon reflects the different sensitivities of the two metrics on an unbounded state space: the total variation distance measures discrepancies in probability mass, whereas the Wasserstein distance additionally captures the spatial cost of transporting that mass. Consequently, a small amount of probability mass escaping toward increasingly remote states may become negligible in total variation distance while remaining significant in the Wasserstein distance. Therefore, the regularity condition represents a genuine qualitative boundary separating distinct stochastic relaxation mechanisms, rather than merely a technical assumption.

We have also characterized the stochastic dynamics in the two regimes where the corresponding deterministic trajectories diverge using Lyapunov function techniques. When $\lambda_{\max}>0$, every interior state of the stochastic system is transient. Starting from an interior state, the process either escapes to infinity or, when the boundary is accessible, is eventually absorbed at the boundary of the state space. When $\lambda_{\max}=0$ but the zero eigenvalue of $A$ is not semisimple or $b$ is not orthogonal to the left nullspace of $A$, every interior state is either transient or null recurrent. In both cases, no closed irreducible positive recurrent class can contain an interior state, and any stationary distribution, if it exists, must be supported on the boundary. These results demonstrate that deterministic divergence does not necessarily imply stochastic divergence: extinction and boundary absorption may instead dominate the long-time behavior of the stochastic system.

Several important questions remain open. First, in the irregular critical regime, it would be desirable to establish quantitative subexponential convergence rates in total variation and to identify structural criteria that determine these rates. Second, in the divergent regimes, a more refined analysis of the boundary dynamics could characterize the probabilities of escape and absorption, classify the closed classes contained in the boundary, and determine the conditions under which the time-dependent distribution converges to a boundary-supported limiting distribution. Third, although the present classification applies to linear reaction networks, many reaction systems are only asymptotically linear or involve nonlinear reactions. Extending the spectral and coupling framework developed here to broader classes of stochastic reaction networks would provide a natural direction for future research.

\section*{Acknowledgements}
B. W. acknowledges support from the National Natural Science Foundation of China (NSFC) under grant No. 12601957. C. J. acknowledges support from NSFC under grant Nos.~12526215, T2541031, and 12271020. During the preparation of this manuscript, the authors used ChatGPT to improve the exposition and presentation and to assist in reviewing the internal consistency of selected proofs. The authors independently checked the mathematical arguments and reviewed all AI-assisted revisions. The authors assume responsibility for all content.

%
%
%
%
%
%

\setlength{\bibsep}{5pt}
\footnotesize\bibliographystyle{nature}
\bibliography{myref}

\end{document}